\documentclass[12pt]{article}

\usepackage{amssymb,amsmath,latexsym,graphics, supertabular}
\usepackage[enableskew]{youngtab}
\usepackage{amsbsy}
\usepackage{amscd}
\usepackage{verbatim}
\usepackage{amsfonts}
\usepackage{amsthm}
\usepackage{times}
\usepackage{mathrsfs}
\usepackage{verbatim}
\usepackage{graphicx,subfigure}
\usepackage[colorlinks]{hyperref}
\usepackage{xypic}
\usepackage{txfonts}
\usepackage[english]{babel}
\usepackage[T1]{fontenc}
\usepackage[utf8]{inputenc}
\usepackage{tikz}

\usepackage{fullpage}
\usepackage{authblk}

\usepackage{epstopdf}

\usetikzlibrary{backgrounds,fit,decorations.pathreplacing}
\usetikzlibrary{arrows,shapes,positioning}
\usetikzlibrary{calc,through,backgrounds,chains}

\usetikzlibrary{arrows,shapes,snakes,automata,backgrounds,petri}

\newenvironment{pf*}[1]{\proof[#1]}{\endproof}
\newtheorem{Theorem}[equation]{Theorem}
\newtheorem{Corollary}[equation]{Corollary}
\newtheorem{Lemma}[equation]{Lemma}
\newtheorem{Proposition}[equation]{Proposition}

\theoremstyle{definition}
\newtheorem{Definition}[equation]{Definition}
\newtheorem{Example}[equation]{Example}

\newtheorem{Conjecture}[equation]{Conjecture}

\theoremstyle{remark}

\newtheorem{Question}[equation]{Question}
\newtheorem{Remark}[equation]{Remark}

\numberwithin{equation}{section}
\newcommand{\Z}{{\mathbb Z}}
\newcommand{\Q}{{\mathbb Q}}
\newcommand{\R}{{\mathbb R}}

\newcommand{\N}{{\mathbb N}}
\newcommand{\spinc}{$\mathrm{Spin}^c$}

\newcommand{\mf}[1]{\mathfrak{#1}}

\newcommand{\mb}[1]{\mathbb{#1}}

\everymath{\displaystyle}

\newcommand{\lf}{\left\lceil}
\newcommand{\rf}{\right\rceil}

\title{Calculating Heegaard-Floer Homology by Counting Lattice Points in Tetrahedra}

\author[1]{Mahir Bilen Can}
\author[2]{\c{C}a\u gr\i \; Karakurt}

\affil[1]{mcan@tulane.edu, Tulane and Yale Universities}
\affil[2]{karakurt@math.utexas.edu, University of Texas, Austin}

\date{\today}

\begin{document}

\maketitle

\begin{abstract}
We introduce a notion of complexity for Sefiert homology spheres by establishing a correspondence 
between lattice point counting in tethrahedra and the Heegaard-Floer homology.
This complexity turns out to be equivalent to a version of Casson invariant and it is 
monotone under a natural partial order on the set of Seifert homology spheres. 
Using this interpretation we prove that there are finitely many Seifert homology spheres with a 
prescribed Heegaard-Floer homology. 
As an application, we characterize L-spaces and weakly elliptic manifolds among Seifert homology spheres.  
Also, we list all the Seifert homology spheres up to complexity two.
\end{abstract}


{\em MSC-2010: 57R58, 57M27, 14J17, 05A15}

\section{Introduction}

Heegaard-Floer homology, introduced by Ozsv\'ath and Szab\'o in \cite{OS4} and \cite{OS5}, 
is a prominent invariant for $3$-manifolds. 
The goal of our article is to explore Heegaard-Floer homology from a combinatorial point of view in the special case of 
Seifert fibered homology spheres. 
Although it is more geometric than similar theories, such as Donaldson, or Seiberg-Witten theories, the definition of 
Heegaard-Floer homology involves a count of a certain moduli 
space of holomorphic disks into a symmetric product of a surface, which is in general a challenging analytical problem.  
On the other hand for a certain class of manifolds, 
namely plumbed manifolds with at most one bad vertex, works of Ozsv\'ath and Szab\'o \cite{OS1}, and Nemethi \cite{N} 
show that the calculation of Heegaard-Floer homology is a purely combinatorial  problem. 
This class of manifolds is relatively small, but it is still large enough to include all 
Seifert fibered spaces (over $S^2$).

In \cite{N}, for a fixed plumbed $3$-manifold, Nemethi finds an explicit algorithm whose output determines 
the Heegaard-Floer homology completely. 
%Aforementioned paper \cite{OS1} of Ozsv\'ath and Szab\'o contains an alternative algorithm.
An alternative algorithm is described in \cite{OS1} by Ozsv\'ath and Szab\'o.
%See  \cite{OS1} for an alternative algorithm. 
However, computing Heegaard-Floer homology for  infinite families of $3$-manifolds 
seems to be a formidable combinatorial problem for one has to determine all the local maxima and 
minima of  infinite families of sequences which simultaneously solve an infinite family of non-homogeneous 
recurrence relations. See \cite{BN}, \cite{N2}, and \cite{Tw} for some particular cases where this problem is handled.

To elaborate on the problem mentioned in the previous paragraph, let us briefly review Nemethi's method 
in the simplest case, where the $3$-manifold is a Seifert homology sphere.  
To this end, let $p_1,\dots,p_l$ be a list of pairwise relatively prime integers such that $1<p_1<p_2<\dots<p_l$. 
We denote by $\varSigma(p_1,p_2,\dots,p_l)$ the Seifert fibered 3-manifold over $S^2$ with $l$ singular fibers 
whose Seifert invariants are given by $(e_0,(p_1' ,p_1),(p_2',p_2),\dots,(p_l',p_l))$. 
Hence, $(x_0,x_1,\dots,x_l)=(e_0,p_1',p_2',\dots,p_l')$ is the unique solution to the Diophantine equation
\begin{equation}\label{e:brieskorn}
x_0p_1p_2\cdots p_l+x_1 p_2\cdots p_l+p_1x_2\cdots p_l+\cdots+p_1p_2\cdots x_l=-1,
\end{equation}
where $1\leq x_i \leq p_i-1$, for  $i=1,2,\dots,l$. 
Equation (\ref{e:brieskorn}) guarantees that  $\varSigma(p_1,p_2,\dots,p_l)$ has  trivial first homology, so it is an integral 
homology sphere. 
To calculate Heegaard-Floer homology of $\varSigma(p_1,p_2,\dots,p_l)$, 
we consider the sequence $\tau :\mathbb{N}\to \mathbb{Z}$ defined by the recurrence 
\begin{equation}\label{eq:rectau}
\tau (n+1)=\tau (n) + 1 + |e_0|n - \sum_{i=1}^{l}\left \lceil \frac{np_i'}{p_i}  \right \rceil 
\end{equation}
with the given initial condition $\tau (0)=0$. Here $\lceil y \rceil$ represents the minimum integer larger than $y$. 
We say that $\tau(n_0)$ is a \emph{local maximum} of $\tau$, if there exist integers $a,b$ such that $a<n_0<b$ with 
$\tau(a)<\tau (n_0) >\tau (b)$, and $\tau$ is monotone increasing on the interval $[a,n_0]$ and monotone decreasing on $[n_0,b]$. 
\emph{Local minimum} values of $\tau$ are defined similarly. 
It turns out that, up  to a degree shift, the Heegaard-Floer homology is determined by the subsequence $\tau'$ 
of $\tau$ consisting of all local minima and local maxima .

In our first  result we analyze the difference term in (\ref{eq:rectau}) in order to understand the local extrema of $\tau$ . 
For notational convenience we focus our attention to Brieskorn spheres, which are by definition the 
Seifert homology spheres with three singular fibers ($l=3$). Nevertheless, most of our arguments are adaptable for studying
arbitrary number of singular fibers with some notational changes. See Theorem \ref{theo:mainmore}.

\begin{Theorem}\label{theo:main}
Let $(p,q,r)$ be a  triple of pairwise relatively prime  integers with $1<p<q<r$.  
Define $\Delta: \mathbb{N}\to \mathbb{Z}$
$$
\Delta (n)= 1 + |e_0|n - \left \lceil \frac{np'}{p} \right \rceil - \left \lceil \frac{nq'}{q} \right \rceil - \left \lceil \frac{nr'}{r} \right \rceil,
$$
where $(e_0,p',q',r')$ is defined by
$$
e_0pqr+p'qr+pq'r+pqr'=-1
$$
with $0\leq p'\leq p-1$, $0\leq q'\leq q-1$, $0\leq r'\leq r-1$. Define the constant
$$
N_0=pqr-pq-qr-pr.
$$
Suppose $(p,q,r) \neq  (2,3,5)$. Then the following holds.
\begin{enumerate}
	\item $N_0$ is a positive integer.
	\item $\Delta (n) \geq 0 $,  for all $n>N_0$.
	\item $\Delta(n)=-\Delta(N_0-n)$,  for all $n$ with $0\leq n\leq N_0$.
	\item  $\Delta(n)\in \{-1,0,1\}$, for all $n$ with $0\leq n \leq N_0$.
	\item For  $0\leq n \leq N_0$, one has $\Delta (n)=1$ if and only if 
	$n$ is an element of the numerical semigroup $G(pq,pr,qr)$ minimally generated by $pq$, $qr$, and $pr$.
	(We consider 0 as an element of the semigroup, hence it is always true that $\Delta(0)=1$.)
\end{enumerate}
If $(p,q,r) =  (2,3,5)$, then $\Delta(n) \geq 0$ for all $n\in \mathbb{N}$.
\end{Theorem}
As we justify later, the above theorem provides us with a fast and  practical means for calculation of the Heegaard-Floer homology 
of a Brieskorn sphere. 
More importantly it gives a partial answer to the realization problem which we explain now. 
%in the sequel. 

Let $U$ be a formal variable and $Y$ be a closed, oriented $3$-manifold. The Heegaard-Floer homology
of $Y$ is a $\Z_2$-graded $\Z[U]$-module with a decomposition 
$HF^+ (Y) \simeq \oplus_{\mf{s}} HF^+(Y,\mathfrak{s})$ into $\Z_2$-graded $\Z[U]$-submodules 
indexed by the \spinc structures on $Y$.
In the case of integral homology spheres the decomposition simplifies; 
there is unique \spinc structure, and the $\Z_2$ grading lifts to a $\Z$-grading
such that $U$ has degree $-2$. 
Therefore, the following question becomes natural:
\begin{Question}\label{q:which}
Which $\mathbb{Z}$-graded $\mathbb{Z}[U]$-module can be realized as the 
Heegaard-Floer homology of a Seifert homology sphere $Y$?
\end{Question}

It is known for integral homology spheres that $HF^+$ decomposes into $\Z[U]$-submodules as follows
$$
HF^+(Y) \simeq \mathcal{T}^+_{(d)}\oplus HF_{\mathrm{red}}(Y),
$$ 
where $\mathcal{T}^+_{(d)}$ is a copy of $\mathbb{Z}[U,U^{-1}]/U\cdot \mathbb{Z}[U]$ on which we impose a 
grading so that the minimal degree is $d$. 
Furthermore, $HF_{\mathrm{red}}(Y)$ is finitely generated.
If $Y$ is a Seifert homology sphere oriented so that it bounds a positive definite plumbing, then $HF^+(Y)$ is 
supported at even degrees only.

Let us illustrate how Theorem \ref{theo:main} is useful in the processes of solving Question \ref{q:which}. 
Recall that a $3$-manifold $Y$ is said to be an {\em $L$-space}, 
if it is a rational homology sphere, and $HF_\mathrm{red}(Y,\mathfrak{s})=0$ for every \spinc structure $\mathfrak{s}$ , \cite{OS2}.  

\begin{Conjecture}\label{conj:lspace}
If an irreducible integral homology sphere $Y$ is an $L$-space, then $Y$ is either the $3$-sphere, or the 
Poincar\'e homology sphere $\varSigma (2,3,5)$ with either orientation.
\end{Conjecture} 
This conjecture is verified for Seifert homology spheres independently by Rustamov \cite{R} and Eftekhary  \cite{E}. 
There is also an implicit proof of the same statement when one combines the results of \cite{BRW} with \cite{LS}. 
Here we give an alternative, elementary proof.
 
\begin{Theorem}\label{theo:lspace}(\cite{R}, \cite{E})
If a Seifert homology sphere $Y$ is an $L$--space, then $Y$ is homeomorphic to $S^3$ or $\pm\varSigma(2,3,5)$.
\end{Theorem}
Among Seifert manifolds, $L$-spaces are precisely those 3-manifolds with $\Delta(n) \geq 0$ for all $n\geq \N$. 
Therefore, Theorem \ref{theo:main} is sufficient to prove Theorem \ref{theo:lspace} in the case of three singular fibers. 
The case of arbitrarily many singular fibers follows from an extension of our theorem to that setting.

Above results suggest that the sum of negative values of the $\Delta$ function is a significant quantity for it defines a 
kind of ``complexity'' for the Heegaard-Floer homology. Indeed, what we observe above is that the complexity 
$0$ Seifert manifolds are precisely the $L$-spaces. Therefore, our next definition is meaningful.

\begin{Definition}\label{def:kappa}
Let $p_1,\dots,p_l$ be a list of relatively prime integers such that $1<p_1<\cdots < p_l$. 
For $\tau(n)$ as defined in (\ref{eq:rectau}),
put $\Delta(n)=\tau(n+1)-\tau(n)$. We define 
$$
\kappa(p_1,p_2,\dots,p_l)=\left | \sum_{i=0}^\infty \mathrm{min} \{0, \Delta(n) \} \right |.
$$
\end{Definition}

It follows from \cite{N} that if two Heegaard-Floer homology groups $HF^+(\varSigma (p_1,p_2,\dots,p_l))$ and 
$HF^+(\varSigma (q_1,q_2,\dots,q_l))$ are isomorphic, then the corresponding  \emph{kappa} invariants 
$\kappa(p_1,p_2,\dots,p_l)$ and $\kappa(q_1,q_2,\dots,q_l)$ are equal. 
The converse does not hold in general, however, as we show, there are only finitely many isomorphism types of 
$\mathbb{Z}[U]$-modules which might appear as the Heegaard-Floer homology of a Seifert homology sphere with a prescribed $\kappa$.

\begin{Theorem}\label{theo:finite}
For any positive integer $k$, there exists finitely many tuples $(p_1,p_2,\dots,p_l)$ such that $\kappa (p_1,p_2,\dots,p_l)=k$, 
where $p_1,p_2,\dots,p_l$ are pairwise relatively prime, and $1<p_1<p_2<\dots<p_l$. 
Consequently, there exist at most finitely many Seifert homology spheres with prescribed Heegaard-Floer homology.
\end{Theorem}  
By contrast, we should mention that it is possible to find many infinite families of irreducible integral homology spheres with isomorphic 
Heegaard-Floer homology. See, for example, Proposition 1.2 of \cite{AK}.

The most important ingredient in the proof of Theorem \ref{theo:finite} is the monotonicity property of $\kappa$ with respect to a 
partial ordering on the set of tuples. We prove this property in propositions \ref{prop:finite}, \ref{prop:kappa2}, and \ref{prop:kappa3}. 
These results together with  sufficient computational power, allows one to list all the graded $\mathbb{Z}[U]$-modules that 
could appear as the Heegaard-Floer homology of a Seifert homology sphere up to a given complexity. 
In the following theorem, we give this list up to $\kappa=2$.

\begin{Theorem}\label{theo:list}
Table \ref{tab:kappa} contains the list of all graded $\mathbb{Z}[U]$-modules that are isomorphic to a Heegaard-Floer 
homology for some Seifert homology sphere with $\kappa \leq 2$. 
Additionally, for each such $\Z[U]$-module $M$, the table contains the list of all Seifert homology spheres whose Heegaard-Floer homology is $M$.
\begin{table}[htp]
\begin{center}
\label{tab:kappa}
\begin{tabular}{|c|c|c|c|}
\hline
Brieskorn Sphere $Y$ &$\kappa$ & $d(-Y)$ & $HF^+(-Y)$ \\
\hline
$S^3$ & $0$ & $0$ & $\mathcal{T}^+_{(0)}$\\
\hline
$\varSigma (2,3,5)$ & $0$ & $-2$ & $\mathcal{T}^+_{(-2)}$\\
\hline
$\varSigma (2,3,7)$ & $1$ & $0$ & $\mathcal{T}^+_{(0)}\oplus \mathbb{Z}_{(0)}$\\
\hline
$\varSigma (2,3,11)$ & $1$ & $-2$ & $\mathcal{T}^+_{(-2)}\oplus \mathbb{Z}_{(-2)}$\\
\hline
$\varSigma (2,3,13), \; \varSigma (2,5,7), \; \varSigma (3,4,5) \;$ & $2$ & $0$ & $\mathcal{T}^+_{(0)}\oplus \mathbb{Z}_{(0)}\oplus \mathbb{Z}_{(0)}$\\
\hline
$\varSigma (2,3,17), \; \varSigma (2,5,9)$ & $2$ & $-2$ & $\mathcal{T}^+_{(-2)}\oplus \mathbb{Z}_{(-2)}\oplus \mathbb{Z}_{(-2)}$\\
\hline
\end{tabular}
\caption{Seifert homology spheres with $\kappa \leq 2$.}
\label{tab:kappa}
\end{center}
\end{table}

\end{Theorem}
Note that Theorem \ref{theo:lspace} is a special case of Theorem \ref{theo:list} for which $\kappa=0$.
We should note also that the $HF^+ (-Y)$ completely determines the Heegaard-Floer homology of positively oriented 
$Y$, so there is no loss of information in Theorem \ref{theo:list}.

Using a relation between the Euler characteristic of Heegaard-Floer homology and the Casson invariant, we 
relate $\kappa$ to other well known invariants of Seifert homology spheres. 
Recall that every Seifert fibered space is the boundary of a $4$-manifold that is plumbing of disk bundles over spheres, 
where the plumbing is done according to a negative definite star shaped weighted tree. 
Such a plumbing configuration is unique up to blow-up and blow-down. 
Suppose we fix one such plumbing and let $s$ denote the number of its vertices. Also, let $K$ denote its canonical cohomology class. 
Then the number $K^2+s$ is invariant under blow-up and blow-down, so it defines an invariant of the Seifert fibered space.

\begin{Proposition}\label{prop:kappa} 
For a Seifert homology sphere $Y=\varSigma (p_1,p_2,\dots, p_l)$,  the following equality holds
$$\kappa (p_1,p_2,\dots, p_l)= \lambda (-Y) - (K^2+s)/8,$$
\noindent where $\lambda(-Y)$ is the Casson invariant of $-Y$ \cite{AM}, normalized so that the 
Poincar\'e homology sphere $\varSigma(2,3,5)$ oriented as the boundary of negative 
$E_8$ plumbing satisfies $\lambda(-\varSigma (2,3,5))=-1$. 
\end{Proposition}

In our next result we observe a  remarkable connection between $\kappa$, numerical semigroups and lattice points in tetrahedra.  
Casson invariants of Brieskorn spheres have similar interpretations. See \cite{FS}. 
We state it for the special case of $3$-singular fibers here. 
There is also a more technical statement that works for arbitrary number of singular fibers 
which we state in Theorem \ref{theo:genmono}.

\begin{Theorem}\label{theo:lattice}
Given a Brieskorn sphere $\varSigma(p,q,r)$ with defining integers $1<p<q<r$, its kappa invariant 
$\kappa (p,q,r)$ is equal to the number of lattice points inside 
the tetrahedron with vertices $(0,0,0)$, $(N_0/pq,0,0)$,  $(0,N_0/pr,0)$, and $(0,0,N_0/qr)$, where $N_0=pqr-pq-pr-qr$. 
In other words, $\kappa (p,q,r)$ equals the cardinality of the set $G\cap[0,N_0]$, 
where $G=G(pq,pr,qr)$ is the numerical semigroup generated by $0$, $pq$, $qr$, and $pr$.
\end{Theorem}

We push our techniques further to study a class of Brieskorn spheres that has a simple Heegaard-Floer homology. 
The following definition is due to Nemethi \cite{N}.

\begin{Definition}
A rational homology sphere $Y$ which is the boundary of a negative definite plumbing tree with at most one bad vertex 
is said to be \emph{weakly elliptic}, if its Heegaard-Floer homology in the canonical \spinc structure is of the form 
$T^+_{(d)}\oplus(\mathbb{Z}_{(d)})^l$ for some $l\geq 1$ and some even integer $d$. 
\end{Definition}
It is shown in Proposition 6.5 of \cite {N} that, if $Y$ weakly elliptic, then it is the link of a weakly elliptic singularity. 
Our next result gives the complete list of weakly elliptic Seifert homology spheres.

\begin{Theorem}\label{T:complete list}
A Brieskorn sphere $\varSigma = \varSigma(p,q,r)$ is weakly elliptic if and only if 
$(p,q,r)$ is equal to one of the following triplets; $(3,4,5),(2,5,7),(2,5,9)$, or $(2,3,r)$ with $gcd(6,r)=1$ and $r> 5$. 
There are no weakly elliptic Seifert homology spheres with more than three singular fibers. 
\end{Theorem}

To further analyze the relationship between lattice point counting and $\tau$ of $\varSigma(p_1,\dots,p_l)$, 
we compare their generating functions. Our computations show that  the generating function of 
$\tau (n)$ is a rational function, similar to the generating function of the sequence counting the lattice points on the hyperplane 
$p_1x_1+p_2x_2+\cdots+p_lx_l=n$ that lie in the first orthant of $\R^l$. 
For the sake of space, here we write only the simplified version of the generating function of $\tau(n)$. 
See Theorem \ref{T: First Theorem} for its explicit form.

\begin{Theorem} \label{C:simplified}
The generating function $F(x)=\sum_n\,\tau (n)x^n$ is given by 
$$
F(x) = \frac{ G(x) }{(1-x^{p_1})(1-x^{p_2})\cdots(1-x^{p_l})},
$$
where $G(x)$ is a polynomial in $x$ with degree less than or equal to $p_1 + p_2 +\dots + p_l-1$. 
Furthermore, if $|e_0 | \neq 1$, then the  degree of $G(x)$ is exactly $p_1 + p_2 +\dots + p_l-1$. 
\end{Theorem}

The organization of our paper is as follows.  
In Section \ref{s:prelim} we review Nemethi's method and see how Heegaard-Floer homology 
is calculated from the $\tau$-function. We analyze the $\Delta$-function  
in Section \ref{s:andel}, and prove therein Theorem \ref{theo:main} and Theorem \ref{theo:lattice}. 
We extend our results to arbitrary number of singular fibers in Section \ref{s:more}. 
Theorems \ref{theo:lspace}, \ref{theo:finite}, and \ref{theo:list} are proved in Section \ref{s:topoappl}. 
In Section \ref{s:weakly elliptic}, we characterize weakly elliptic Brieskorn spheres interms of 
certain numerical semigroups and prove Theorem \ref{T:complete list}.
Finally, we conclude our paper by calculating the generating function of $\tau(n)$ in Section \ref{s:genfunc}.

\section{Graded Roots and Heegaard-Floer Homology}
\label{s:prelim}

Here we review the definition of a ``graded root,'' and discuss its basic properties.
For more information and background, we recommend \cite{N} and Section $2$ of \cite{AK}.

\begin{Definition}
A {\em graded root} is a pair $(R,\chi)$, where $R$ is an infinite tree, and $\chi$ is an integer valued function defined 
on the vertex set $V=V(R)$ of $R$ satisfying the following properties. 
\begin{enumerate}
	\item $\chi(u)-\chi(v)=\pm 1$, if there is an edge connecting $u$ and $v$.
	\item $\chi(u)>\mathrm{min}\{v,w\}$, if there are edges connecting $u$ to $v$, and $u$ to $w$.
	\item $\chi$ is bounded below.
	\item $\chi^{-1}(k)$ is finite for every $k$.
	\item $|\chi^{-1}(k)|=1$ for $k$ large enough.
\end{enumerate}
\end{Definition}

In Figure \ref{fig:GradedRoot} we give an example of a graded root, where the infinite tree 
$R$ is drawn on left, and the function $\chi$ is obtained from the heights of the vertices.

\begin{figure}[htp]
\begin{center}
\begin{tikzpicture}[scale=.65]

\begin{scope}[xshift=-2cm,yshift=0]

\path[draw,dashed,thick] (0,4.5) -- (0,3);   
\path[draw,-,thick] (0,3.5) -- (0,-3);   

\path[draw,-,thick] (0,-1) -- (1,-2);   
\fill[draw=blue,fill=blue, very thick] (1,-2) circle(3pt); 

\path[draw,-,thick] (0,0) -- (2,-2);   
\fill[draw=blue,fill=blue, very thick] (1,-1) circle(3pt); 
\fill[draw=blue,fill=blue, very thick] (2,-2) circle(3pt); 

\path[draw,-,thick] (0,0) -- (-2,-2);   
\fill[draw=blue,fill=blue, very thick] (-1,-1) circle(3pt); 
\fill[draw=blue,fill=blue, very thick] (-2,-2) circle(3pt); 

\foreach \y in {-3,...,3} 
{ 
\fill[draw=blue,fill=blue, very thick] (0,\y) circle(3pt); 
\path[draw,dashed] (3,\y) -- (12,\y); 
}

\end{scope}

\begin{scope}[xshift=1.5cm]

%\path[draw,-] (9,4.5) -- (9,-4);   
\node at (10,3) {4};
\node at (10,2) {3};
\node at (10,1) {2};
\node at (10,0) {1};
\node at (10,-1) {0};
\node at (10,-2) {$-1 $};
\node at (10,-3) {$-2 $};
\end{scope}

\end{tikzpicture}
\caption{A graded root.}
\label{fig:GradedRoot}
\end{center}
\end{figure}
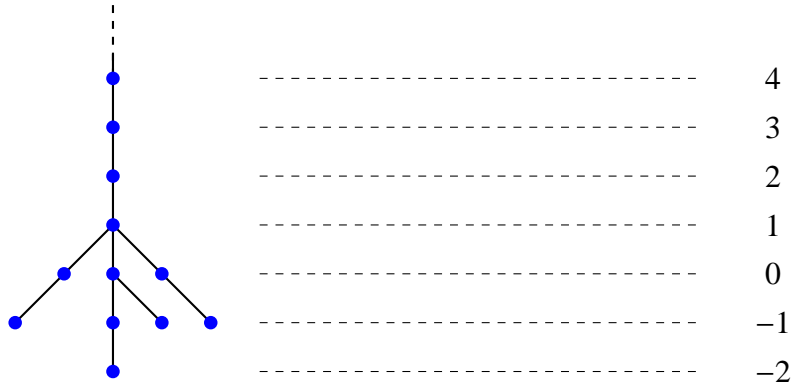

The branches of $R$ are enumerated from left to right, and the vertex at the bottom of the $i$-th 
branch is called the {\em $i$-th vertex}. Let us denote by $a_i$ the value of $\chi$ on the $i$-th vertex, 
and denote by $b_i$ the value of $\chi$ at the branching vertex connecting $i$-th branch to the $(i+1)$-th branch.
Then the data of $(R,\chi)$ is encoded in the sequence $[a_1,b_1,a_2,b_2,\dots,a_{n-1},b_{n-1},a_n]$.
For example, the graded root given in Figure \ref{fig:GradedRoot} is represented by the sequence $[-1,1,-2,0,-1,1-1]$.

Conversely, any sequence  $[a_1,b_1,a_2,b_2,\dots,a_{n-1},b_{n-1},a_n]$ satisfying  
$b_i>\mathrm{max}\{a_i,a_{i+1} \}$, for $i=1,\dots,n-1$ determines a graded root. 
For a given sequence $\tau$ with this property, we denote the corresponding graded root by $(R_\tau,\chi_\tau)$.

The nomenclature of ``graded root'' is explained by the natural correspondence between graded roots and 
graded $\Z[U]$-modules.
Consider the free $\mathbb{Z}$-module generated by the vertex set $V$. 
The $U$ action is described as follows. 
For a given vertex $v$, $U\cdot v$ has a summand supported at the vertex $w$, 
if there is an edge connecting $v$ to $w$, and $\chi(v)>\chi (w)$. 
Then the action of $U$ is extended by linearity. 
Finally, the grading is determined by the requirement that every vertex $v$ has degree $2\chi (v)$. 
Given a graded root $(R,\chi)$,  we denote the associated  $\mathbb{Z}[U]$-module by $\mathbb{H}(R,\chi)$. 
For example, for the graded root given in Figure \ref{fig:GradedRoot}, the associated $\mathbb{Z}[U]$-module is isomorphic to 
$\mathcal{T}^+_{(-4)}\oplus \mathcal{T}^1_{(-2)}\oplus \mathcal{T}^1_{(-2)}\oplus \mathbb{Z}_{(-2)}$. 
Here, we use the following notation: 
$\mathcal{T}^+_{(d)}=\mathbb{Z}[U,U^{-1}]/\langle U \rangle$, and $\mathcal{T}^n_{(d)}=\mathbb{Z}[U]/\langle U^{n+1} \rangle$; 
both groups are graded so that $U$ has degree $-2$, and the minimal degree is $d$.

Fix a Seifert homology sphere $\varSigma (p_1,\dots,p_l)$, and let $\tau$ denote the sequence defined recursively as in (\ref{eq:rectau}). 
It is known that $\tau (n)$ is an increasing function of $n$, for all sufficiently large $n \gg 0$. 
It follows that the subsequence consisting of local minima and local 
maxima of $\tau$ is a finite sequence. By abuse of notation, we denote this finite subsequence by $\tau$, also. 
Now consider the graded root given by $\tau$ and its $\mathbb{Z}[U]$-module $\mathbb{H}(R_\tau,\chi_\tau)$. 
It turns out that, up to a global degree shift 
the Heegaard-Floer homology of $\varSigma (p_1,\dots,p_l)$ is isomorphic to $\mathbb{H}(R_\tau,\chi_\tau)$.

The degree shift is calculated as follows. Let $X$ denote the $4$-manifold $X$ bounding $\varSigma (p_1,\dots,p_l)$, 
which is a star shaped plumbing of certain disk bundles over $2$-sphere with a negative definite intersection form. 
The second homology of $X$ has a natural basis $e_0,e_1,\dots,e_{s-1}$ consisting of base spheres. 
Here, $e_0$ corresponds to the central vertex in the plumbing graph, and $s$ is equal to the total number of vertices. 
The canonical $2$-cohomology class $K$ is defined by the requirement that $K(e_i)=-e_i\cdot e_i-2$. 
Then, the desired degree shift is given by $-(K^2+s)/4$. 

An alternative approach utilizes the ``Dedekind sums'' for computing the degree shift. 
The {\em Dedekind sum}, $s(p,q)$ is calculated recursively by setting $s(1,1)=0$ and repeatedly applying the reciprocity law
$$
s(p,q)+s(q,p)=-\frac{1}{4}+\frac{1}{12}\left (\frac{p}{q}+\frac{q}{p}+\frac{1}{pq} \right ), 
$$
and using the rule stating that whenever $r\equiv p\mod q$, the equality $s(p,q)=s(r,q)$ holds.

It is shown in \cite{NN} that 
\begin{equation}\label{e:Nemshift}
K^2+s=\epsilon ^2 e +e+5-12\sum_{i=1}^ls(p'_i,p_i),
\end{equation}
where
$$
e=e_0+\sum_{i=1}^l\frac{p'_i}{p_i},\;\;\; \mathrm{and} \;\;\; \epsilon=\left ( 2-l+\sum_{i=1}^l \frac{1}{p_i}\right )\frac{1}{e}.
$$

In conclusion we have the following result. 
\begin{Theorem}\label{theo:nem}(\cite{N})
For any Seifert homology sphere $Y:=\varSigma(p_1,p_2,\dots,p_l)$, 
the Heegaard-Floer homology group $HF^+(-Y)$ is isomorphic to $\mathbb{H}(R_\tau,\chi_\tau)$ 
with a degree shift $-(K^2+s)/4$, where $\tau $ is the sequence defined in (\ref{eq:rectau}), 
and $(R_\tau,\chi_\tau)$ is the associated graded root.  
\end{Theorem}

\begin{Example}
Combining Theorem \ref{theo:nem} and our Theorem \ref{theo:main}, it is now easy to calculate the Heegaard-Floer homology of a
Seifert homology sphere. Let us illustrate this statement on $Y=\varSigma (2,3,11)$.

We have $N_0=5$. 
Consider $G=G(6,22,33)$, the numerical semigroup generated by the integers $6$, $22$, and $33$. 
The only element of $G$ that is contained in the interval $[0,N_0]$ is $0$. 
Therefore Theorem \ref{theo:main} implies  $\Delta(0)=1$, $\Delta (5)=-1$, and $\Delta(n)\geq 0$ for all $n \neq 0,5$. 
Hence $\tau (n)=\sum_{i=0}^{n-1} \Delta(i)$ has two local minimum  values (both of which are equal to $0$) and one local maximum value 
(which is equal to $1$). Let $(R_\tau,\chi_\tau)$ denote the graded root associated with the sequence $\tau=[0,1,0]$. 
Then $\mathbb{H}(R_\tau,\chi_\tau)=\mathcal{T}^+_{(0)}\oplus\mathbb{Z}_{(0)}$. 

We need to calculate the degree shift $-(K^2+s)/4$. 
It follows from (\ref{e:brieskorn}) that the Seifert invariants of $\varSigma(2,3,11)$ are given by 
$$
(e_0,(p',p),(q',q),(r',r))=(-2,(1,2),(2,3),(9,11)).
$$ 
\noindent We calculate the terms appearing in (\ref{e:Nemshift}), and see that $e=-1/66$, $\epsilon=5$. 
The Dedekind sums are calculated by repeatedly applying the reciprocity law:
\begin{align*}
s(1,2)&=0,\\
s(2,3)&=-1/18,\\
s(9,11)&=-5/22.
\end{align*}
Using these values in (\ref{e:Nemshift}), the degree shift is calculated to be $-(K^2+s)/4=-2$.
Theorem \ref{theo:main} says that $HF^+(-\varSigma(2,3,11))=\mathcal{T}^+_{(-2)}\oplus\mathbb{Z}_{(-2)}$.
\end{Example}

\section{Analysis of the Delta Function}\label{s:andel}

In order to determine the positions and values of the local extrema of $\tau$ function, we study its difference term
\begin{equation*}
\Delta(n)=1+e_0n-\left \lceil \frac{np'}{p} \right \rceil -\left \lceil \frac{nq'}{q} \right \rceil -\left \lceil \frac{nr'}{r} \right \rceil. 
\end{equation*}
Our first task is to write $\Delta$ as a quasi-polynomial. To this end, we consider $f: \Q \rightarrow [0,1]$ defined 
by 
\begin{equation}\label{d:f}
f(x):= \lceil x \rceil -x.
\end{equation}
\begin{Lemma} \label{l:deltaprop}
Given relatively prime integers $a$ and $m$, the sequence 
$$
g(n)=f(na/m)=\lceil na/m \rceil-na/m
$$ 
is periodic with period  $m$. 
Moreover, the finite sequence $(mg(0), mg(1),\dots,mg(m-1))$ is the same as the orbit of 
$m-a$ in the additive group $\mathbb{Z}/m\mathbb{Z}$. 
Consequently, for every $s \in \{0,\dots m-1\}$ there exists unique $n$ such that $0\leq n \leq m-1$ and $f(na/m)=s/m$. 
\end{Lemma}
\begin{proof}
Writing $n=pm+r$ we see that $g(n)=g(r)$, establishing the periodicity of $g$. 
For the second part it suffices to observe that $g(n)$ is the fractional part of $n(m-a)/m$. 
Indeed, 
$$
\frac{n(m-a)}{m} - \left \lfloor \frac{n(m-a)}{m}\right \rfloor = -\frac{na}{m} -\left \lfloor - \frac{an}{m}\right \rfloor 
= -\frac{na}{m} +\left \lceil \frac{an}{m}\right \rceil.
$$
\end{proof}
We rewrite $\Delta$ accordingly, as follows:
\begin{align}
\Delta(n) &=1-e_0n - \frac{np'}{p}- \frac{nq'}{q} - \frac{nr'}{r} -\left (\left \lceil \frac{np'}{p} \right \rceil - \frac{np'}{p} 
 +\left \lceil \frac{nq'}{q} \right \rceil - \frac{nq'}{q} +\left \lceil \frac{nr'}{r} \right \rceil  - \frac{nr'}{r} \right ) \notag \\
&=1-\frac{e_0pqr+p'qr+pq'r+pqr'}{pqr}n -\left ( f\left (\frac{np'}{p} \right ) +f\left (\frac{nq'}{q} \right ) + f\left (\frac{nr'}{r} \right ) \right ) \notag \\
&=1+\frac{1}{pqr}n -\left ( f\left (\frac{np'}{p} \right ) +f\left (\frac{nq'}{q} \right ) + f\left (\frac{nr'}{r} \right ) \right )\qquad 
\text{(by using\ (\ref{e:brieskorn}))}.
\label{e:deltacompact}
\end{align}

\begin{Remark}
The periodic nature of $\Delta$ is now apparent from (\ref{e:deltacompact}). 
This is suggested by the generating function calculation in Section \ref{s:genfunc}, also. 
In fact, it follows from generating function calculations that $\Delta$ is the sum of a linear polynomial and a periodic function 
(which, in turn, can be written as the sum of three periodic functions). 
\end{Remark}

Equality (\ref{e:deltacompact}) allows us to do the following critical analysis regarding the values of $\Delta$.
\begin{Proposition}\label{prop:delta}
For $0\leq n<pqr$, we have $-1 \leq \Delta (n) \leq 1$. Moreover,
\begin{enumerate}
	\item $\Delta (n)=-1$ if and only if $\;\displaystyle  f\left (\frac{np'}{p} \right ) +f\left (\frac{nq'}{q} \right ) + f\left (\frac{nr'}{r} \right ) \geq 2$.
	\item $\Delta (n)=1$ if and only if $\;\displaystyle f\left (\frac{np'}{p} \right ) +f\left (\frac{nq'}{q} \right ) + f\left (\frac{nr'}{r} \right )  \leq 1$.
\end{enumerate}
For $n \geq pqr$, we have $\Delta(n)\geq 0$.
\end{Proposition}

\begin{proof}
Clearly, $\Delta(0)=1$. Suppose $0<n<pqr$. 
Let $A(n)=\;\displaystyle f\left (\frac{np'}{p} \right ) +f\left (\frac{nq'}{q} \right ) + f\left (\frac{nr'}{r} \right )$. 
We have $0<A(n)<3$, and $0<n/pqr<1$. Therefore $\Delta (n)=  1+ n/pqr-A(n)$ satisfies
$$
-2<\Delta(n)<2.
$$
Note that $\Delta (n)$ is integer valued, hence $-1 \leq \Delta(n) \leq 1$ for all $n<pqr$.
If $A(n) \geq 2$ then $\Delta (n) < 0$, so item 1 follows again from the fact that $\Delta$ is integer valued. 
Similarly, if $A(n) \leq 1$ then $\Delta > 0$, so we obtain the second item. 

For $n>pqr$, we have $n/pqr>1$, so $\Delta (n) > -1$ since $A(n)<3$.
\end{proof}

\vspace{1cm}

\begin{proof}[Proof of Theorem \ref{theo:lattice}]
It follows from Proposition \ref{prop:delta} that the number of times $\Delta$ attains -1 is the 
number of triples $(x,y,z) \in \N^3$ satisfying $0\leq x \leq p-1$, $0\leq y \leq q-1$, $0\leq z \leq r-1$
and 
\begin{align}\label{greater than 2}
\frac{x}{p} + \frac{y}{q}+ \frac{z}{r} \geq 2.
\end{align} 
We interpret this number as the number of lattice points in a tetrahedron as follows. 
Inequality (\ref{greater than 2}) is equivalent to $z \geq 2 r -  \frac{rx}{p} - \frac{ry}{q}$. 
Therefore, we are seeking for the number of lattice points inside the prism 
$[0,p-1]\times [0,q-1]\times [0,r-1]$ that lie above the hyperplane $\varGamma:\ z = 2 r -  \frac{rx}{p} - \frac{ry}{q}$.
A straightforward calculation shows that the hyperplane $\varGamma$ intersects $x= p-1$ plane 
along the line $qz+ry = qr+qr/p$.
Similarly, it intersects $y= q-1$ plane along the line $pz+rx = pr+pr/q$. 
On $z=r-1$ plane we have the line $x+y  = pq(r+1)/r$.
We depict a generic picture in Figure \ref{tetrahedron}.

\begin{figure}[htp]
\begin{center}
\begin{tikzpicture}[scale=.7,cap=round,>=latex]

\fill[draw=blue,fill=blue!20!, very thick] (-2.5,-2.5) circle(1.3pt);
\fill[draw=blue,fill=blue!20!, very thick] (5,0) circle(1.3pt);
\fill[draw=blue,fill=blue!20!, very thick] (0,5.5) circle(1.3pt);

\node at (-3.5,-2.5) {{\scriptsize $(p-1,0,0)$}};
\node at (-4.2,-4) {{\scriptsize $x$}};
\node at (5.9,.2) {{\scriptsize $(0,q-1,0)$}};
\node at (8.2,0) {{\scriptsize $y$}};
\node at (0,5.8) {{\scriptsize $(0,0,r-1)$}};
\node at (0,8.2) {{\scriptsize $z$}};

\fill[blue!7,opacity=0.8] (0,0) -- (-2.5,-2.5) -- (-2.5,3) -- (0,5.5) -- cycle; %
\fill[blue!12,opacity=0.8] (0,0) -- (-2.5,-2.5) -- (2.5,-2.5) -- (5,0) -- cycle; %
\fill[blue!17,opacity=0.8] (0,0) -- (0,5.5) -- (5 , 5.5) -- (5,0) -- cycle; %

\path[draw,dotted, thick,  -] (-2.5,3) -- (2.5,3);   
\path[draw,dotted, thick,  -] (2.5,3) -- (5,5.5);   
\path[draw,dotted, thick, -] (2.5,3) -- (2.5,-2.5);   
\path[draw,dotted, thick,  -] (-2.5,-2.5) -- (-2.5,3);   
\path[draw,dotted, thick,  -] (-2.5,-2.5) -- (-2.5,3);   
\path[draw,dotted, thick,  -] (-2.5,3) -- (0,5.5);   
\path[draw,dotted, thick,  -] (5,5.5) -- (0,5.5);   
\path[draw,dotted, thick,  -] (5,5.5) -- (5,0);   
\path[draw,dotted, thick,  -] (5,0) -- (2.5,-2.5);   
\path[draw,dotted, thick,  -] (2.5,-2.5) -- (-2.5,-2.5);   

\path[draw,dotted,->] (0,0) -- (-4,-4);   
\path[draw,dotted,->] (0,0) -- (8,0);   
\path[draw,dotted,->] (0,0) -- (0,8);   

\fill[blue!87,opacity=0.5] (-1,3) -- (2.5,3) -- (4,4.5) -- (-1,3) -- cycle; 
\fill[blue!57,opacity=0.5] (-1,3) -- (2.5,3) -- (2.5,-0.5) -- (-1,3) -- cycle; 
\fill[blue!27,opacity=0.5] (2.5,-0.5) -- (4,4.5) -- (2.5,3)  -- (2.5,-0.5) -- cycle; 
\fill[gray!87,opacity=0.5] (2.5,-0.5) -- (-1,3) -- (4,4.5)  -- (2.5,-0.5) -- cycle; 
\fill[draw=blue,fill=blue!20!, very thick] (2.5,3) circle(1.3pt);
\fill[draw=blue,fill=blue!20!, very thick] (4,4.5) circle(1.3pt);
\fill[draw=blue,fill=blue!20!, very thick] (-1,3) circle(1.3pt);
\fill[draw=blue,fill=blue!20!, very thick] (2.5,-.5) circle(1.3pt);

\path[draw,-, very thick,  -] (2.5,3) -- (4,4.5);   
\path[draw,-, very thick,  -] (2.5,-.5) -- (4,4.5);   
\path[draw,-, very thick,  -] (-1,3) -- (4,4.5);   
\path[draw,-, very thick,  -] (-1,3) -- (2.5,-.5);   
\path[draw,-, very thick,  -] (-1,3) -- (2.5,3);  
\path[draw,-, very thick,  -] (2.5,-.5) -- (2.5,3);   

\end{tikzpicture}
\caption{Tetrahedron corresponding to $\Delta=-1$}
\label{tetrahedron}
\end{center}
\end{figure}
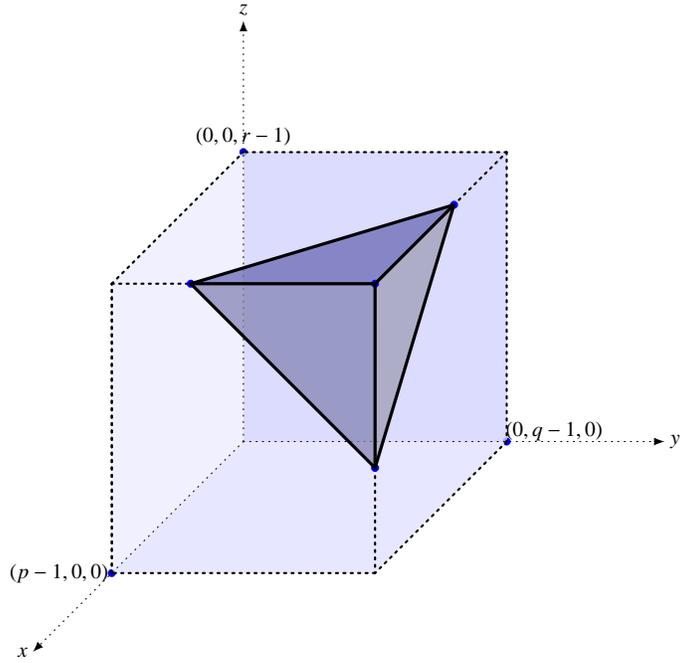

It remains to compute the lattice points in the tetrahedron with the vertices  
$A''=(p-1,q-1,r-1)$, $B''=(p-1, q-1, r/p+r/q)$, $C''=(p-1q/r+q/p,r-1)$ and $D''=(p/r+p/q,q-1,r-1)$.
Shifting the tetrahedron to the origin and simplifying its coordinates give:
\begin{align*}
A' &= (0,0,0)\\
B' &= (0,0, \frac{pq+qr+pr - pqr}{pq})\\
C' &= (0, \frac{pq+qr+pr - pqr}{pr}, 0)\\
D' &= (\frac{pq+qr+pr - pqr}{qr},0,0)
\end{align*}
The affine transformation $x \mapsto -x, y \mapsto -y$ and $z \mapsto -z$ does not alter the number of points in the tetrahedron:
\begin{align*}
A &= (0,0,0)\\
B &= (0,0, -\frac{pq+qr+pr - pqr}{pq})\\
C &= (0, -\frac{pq+qr+pr - pqr}{pr}, 0)\\
D &= (- \frac{pq+qr+pr - pqr}{qr},0,0)
\end{align*}
This proves our claim.

\end{proof}

\begin{Lemma}\label{lem:n0mono}
For positive pairwise relatively prime integers $(p,q,r)$ with $p<q<r$, define $N_0(p,q,r)=pqr-pq-qr-pr$. 
Then $N_0(p_1,q_1,r_1)\leq N_0(p_2,q_2,r_2)$, if $p_1\leq p_2$, $q_1 \leq q_2$, and $r_1\leq r_2$. 
Consequently $N_0(p,q,r)>0$ unless $(p,q,r)=(2,3,5)$.
\end{Lemma}

\begin{proof}
Let $P$ denote the set of triples of positive, pairwise relatively prime integers $(p,q,r)$ with $p<q<r$. 
Consider the partial order on $P$ defined by 
\begin{align}\label{a:natural partial order}
(p_1,q_1,r_1) \leq (p_2,q_2,r_2)\ \text{if}\ p_1\leq p_2,\ q_1 \leq q_2,\ \text{and}\ r_1\leq r_2.
\end{align} 
The triple $(2,3,5)$ is the smallest element of $(P,\leq)$.  
Define $h(x,y,z)=xyz-xy-yz-xz$ for $x\geq 2$, $y\geq 3$, $z\geq 5$, and $x\leq y \leq z$. 
Then one has $\partial h/\partial x> 0$, $\partial h/\partial y> 0$, and $\partial h/\partial z> 0$. 
Therefore $N_0$ respects the partial order $\leq$ on $P$. This proves the first assertion. 
For the second, observe that if $(p,q,r)\in P$ and $(p,q,r)\neq (2,3,5)$, then $(p,q,r)\geq (2,3,7)$, or $(p,q,r)\geq (3,4,5)$. 
Since $N_0(2,3,7)>0$ and $N_0(3,4,5)>0$, we have $N_0(p,q,r)>0$.
\end{proof}

\begin{Lemma}\label{lemm:mono} Suppose $(p,q,r)\neq (2,3,5)$. 
Let $N_0 = pqr-pq - pr - qr>0$. Then $\Delta(N_0)=-1$, and for all $n > N_0$, we have $\Delta (n)\geq 0$ . 
\end{Lemma}

\begin{proof}
The following congruences are easily verified:
\begin{eqnarray*}
p'N_0 &\equiv& 1 \; \mathrm{mod}\; p \\
q'N_0 &\equiv& 1 \; \mathrm{mod}\; q \\
r'N_0 &\equiv& 1 \; \mathrm{mod}\; r 
\end{eqnarray*}
Then Lemma \ref{l:deltaprop} implies
\begin{eqnarray*}
f\left (\frac{N_0p'}{p} \right )&=&\frac{p-1}{p} \\
f\left (\frac{N_0q'}{q} \right )&=&\frac{q-1}{q} \\
f\left (\frac{N_0r'}{r} \right )&=&\frac{r-1}{r}.
\end{eqnarray*}
Substituting these values in equation \ref{e:deltacompact} we get
$$
\Delta (N_0)= 2- \left ( \frac{1}{p} + \frac{1}{q} + \frac{1}{r}\right )-\left ( \frac{p-1}{p} +\frac{q-1}{q}+ \frac{r-1}{r} \right ) =-1.
$$
For the second part, using Lemma \ref{l:deltaprop} once again, we obtain the following estimate:
$$
f\left (\frac{np'}{p} \right ) + f\left (\frac{nq'}{q} \right ) + f\left (\frac{nr'}{r} \right ) \leq \frac{p-1}{p} + \frac{q-1}{q} + \frac{r-1}{r}
$$
for all $n$. 
Suppose $n=N_0+s$, for some $s>0$. Then
$$
\Delta(n)\geq  2- \left ( \frac{1}{p} + \frac{1}{q} + \frac{1}{r}\right )+\frac{s}{pqr}-\left ( \frac{p-1}{p} +\frac{q-1}{q}+ \frac{r-1}{r} \right ) >-1.
$$
Hence, $\Delta (n) \geq 0$ for all $n > N_0$.
\end{proof}

Next, we prove that $\Delta$-function is centrally symmetric with respect to $N_0/2$.
\begin{Lemma}\label{lemm:symm}
Let $N_0=pqr-pq-qr-pr$. For any integer $i$ such that $0 \leq i \leq N_0$, we have 
$$
\Delta (i)=- \Delta (N_0-i).
$$
\end{Lemma}

\begin{proof}
Let $a$, $b$, and $c$ be three integers defined by the conditions
\begin{eqnarray*}
a&\equiv& ip' \; \mathrm{mod} \; p, \;\; 0\leq a \leq p-1,\\
b&\equiv& iq' \; \mathrm{mod} \; q, \;\; 0\leq a \leq q-1,\\
c&\equiv& ir' \; \mathrm{mod} \; r, \;\; 0\leq a \leq r-1.
\end{eqnarray*}
Then by (\ref{e:deltacompact}) we have 
\begin{eqnarray*}
\Delta (i) &=& 1+\frac{i}{pqr}-\left (f\left (\frac{a}{p} \right ) +f\left (\frac{b}{q} \right ) + f\left (\frac{c}{r} \right )\right ), \\
\Delta (N_0-i) &=& 2-\frac{1}{p}-\frac{1}{q}-\frac{1}{r}-\frac{i}{pqr}-\left (f\left (\frac{1-a}{p} \right ) 
+f\left (\frac{1-b}{q} \right ) + f\left (\frac{1-c}{r} \right )\right ).
\end{eqnarray*}
After adding these two equations and plugging the definition of $f$ in, and doing the obvious cancellations, we see that
$$
\Delta (i) + \Delta (N_0-i) = 3- \left ( \left \lceil \frac{1-a}{p} \right \rceil + \left \lceil \frac{a}{p} \right \rceil+ 
\left \lceil \frac{1-b}{q} \right \rceil + \left \lceil \frac{b}{q} \right \rceil+ \left \lceil \frac{1-c}{r} \right \rceil + \left \lceil \frac{c}{r} \right \rceil\right ).
$$
Hence the following lemma finishes the proof

\begin{Lemma}
Let $a$ and $p$ be integers such that $0\leq a \leq p-1$. Then 
$$
\left \lceil \frac{1-a}{p} \right \rceil + \left \lceil \frac{a}{p} \right \rceil=1.
$$
\end{Lemma}

\begin{proof}
If $a\neq 0$, then the first term is $0$ and the other one is $1$. If $a=0$ then the first term is $1$ and the other one is $0$. 
In both cases they add up to $1$.
\end{proof}
\end{proof}

\begin{proof}[Proof of Theorem \ref{theo:main}]
The first four items follow from Lemma \ref{lem:n0mono}, Lemma \ref{lemm:mono}, 
Lemma \ref{lemm:symm}, and Proposition \ref{prop:delta}, respectively. 
The proof of the second part of Lemma \ref{lemm:mono} shows that $\Delta(n) \geq 0$ when $(p,q,r)=(2,3,5)$. 
It remains proving the fifth item.

Let $G=G(pq,pr,qr)$ denote the semigroup generated by 0, $pq$, $pr$, and $qr$. 
If $n\in G$, then $n=aqr+bpr+cpq$ for some $a,b,c\geq 0$. Hence, we see that 
\begin{eqnarray*}
np' & \equiv & -a \; \mathrm{mod}\; p,\\
nq' & \equiv & -b \; \mathrm{mod}\; q,\\
nr' & \equiv & -c \; \mathrm{mod}\; r.\\
\end{eqnarray*}
Let $\widetilde{a},\widetilde{b}$ and $\widetilde{c}$ denote the residues of $a,b,c$ modulo $p,q,r$, respectively. 
Then Lemma \ref{l:deltaprop} implies
$$
f\left (\frac{np'}{p} \right ) +f\left (\frac{nq'}{q} \right ) + f\left (\frac{nr'}{r} \right )=\frac{\widetilde{a}}{p} 
+\frac{\widetilde{b}}{q}+\frac{\widetilde{c}}{r}.
$$
Plugging in (\ref{e:deltacompact}), we get $\Delta (n)\geq 1$. It follows from Proposition \ref{prop:delta} that 
$\Delta(n)=1$, if $n\in G\cap[0,N_0]$.

Next, we show that $G\cap[0,N_0]$ contains all the elements $n\in[0,N_0]$ with $\Delta (n)=1$. 
We prove this by showing that the cardinalities of these two sets are equal. 
Indeed, by the symmetry proven in Lemma \ref{lemm:symm} the number of times $\Delta$ attains $+1$ in $[0,N_0]$ 
is equal to the number of times $\Delta$ attains $-1$ in the same interval. 
On the other hand, we know from Theorem \ref{theo:lattice} that the total number of $-1$'s of 
$\Delta$ is equal to the cardinality of $G\cap[0,N_0]$.
Therefore, the proof is complete.

\end{proof}

Next we establish the monotonicity of $\kappa$ on the set of ordered triples with respect to the natural partial order
$\leq$ defined in (\ref{a:natural partial order}).
\begin{Proposition}\label{prop:finite}
Suppose $(p_1,q_1,r_1)\geq (p_2,q_2,r_2)$, then $\kappa(p_1,q_1,r_1)\geq \kappa(p_2,q_2,r_2)$.
\end{Proposition}

\begin{proof}
In view of Theorem \ref{theo:lattice}, it suffices to show that the tetrahedron $T_1$ corresponding to $(p_1,q_1,r_1)$ 
contains the tetrahedron $T_2$ corresponding to $(p_2,q_2,r_2)$. The edge of $T_1$ on the $x$-axis has length 
$$
l_x(T_1)= r_1\left (1-\frac{1}{p_1}-\frac{1}{q_1} \right )-1,
$$
so, the hypothesis implies that $l_x(T_1)\geq l_x(T_2)$. 
By symmetry, the edges on the $y$-,  and the $z$-axes satisfy the same property, hence, the proof follows.  
\end{proof}

\section{Generalizations}
\label{s:more}

In this section we extend our results to Seifert homology spheres with four or more singular fibers. 
We start with a modified version of Theorem \ref{theo:main}.

%Before we state the modified version of Theorem \ref{theo:main}, let us mention 
%our reasons for treating the general case separately. First, the family of Brieskorn spheres is important 
%enough to warrant an explicit analysis by itself, and the main ideas of our work are more transparent in this situation. 
%Secondly, the Brieskorn sphere case is the initial step of our 
%inductive arguments. Moreover, the proofs of the generalizations rely heavily on this special case. 

\begin{Theorem}\label{theo:mainmore}
For $l\geq 4$, let $(p_1,p_2,\dots,p_l)$ be an $l$-tuple of pairwise relatively prime  integers with $1<p_1<p_2<\dots<p_l$.  
Let $\Delta: \mathbb{N}\to \mathbb{Z}$ denote the function 
$$
\Delta (n)=  1 + |e_0|n - \sum_{i=1}^{l}\left \lceil \frac{np_i'}{p_i}  \right \rceil ,
$$
where $(e_0,p'_1,p'_2,\dots,p'_l)$ is defined by the equation 
$$
e_0p_1p_2\cdots p_l+p_1'p_2\cdots p_l+p_1p_2'\cdots p_l+\cdots+p_1p_2\cdots p_l'=-1,
$$
with $0\leq p'_i\leq p_i-1$, for all $i=1,\dots,l$. Define the constant
$$
N_0=p_1p_2\cdots p_l\left ( \left (l-2 \right )-\sum_{i=1}^l\frac{1}{p_i}\right ) \in \Z_{>0}.
$$
Then
\begin{enumerate}
%	\item $N_0$ is a positive integer.
	\item $\Delta (n) \geq 0 $,  for all $n>N_0$.
	\item $\Delta(n)=-\Delta(N_0-n)$,  for $0\leq n\leq N_0$.
	\item  $\Delta(n)\in \{-(l-2),\dots,-1,0,1,\dots,l-2\}$, for $0\leq n \leq N_0$.
	\item For  $0\leq n \leq N_0$, one has $\Delta (n)\geq1$ if and only if either $n=0$, or 
	$n$ is an element of the numerical semigroup $G$ minimally generated by $p_1p_2\dots p_l/p_i$ for $i=1,2,\dots,l$.
	\item \label{i:theo6} If $n\in G$ is of the form $\displaystyle n=p_1p_2\dots p_l\sum _{i=1}^l \frac{x_i}{p_i}$, 
	then $\Delta (n)=1 + \sum _{i=1}^l \left \lfloor \frac{x_i}{p_i} \right \rfloor$.
\end{enumerate}
\end{Theorem}

We omit the proofs of items 1-3, since they are identical to the case $l=3$, 
except that one needs a generalization of Equation \ref{e:deltacompact} to write $\Delta (n)$ as a linear quasi-polynomial.
\begin{equation}\label{e:delcompgen}
\Delta(n)=1+\frac{n}{p_1\dots p_l}-\sum_{i=1}^lf\left ( \frac{np'_i}{p_i}\right ).
\end{equation}
The proofs of items 4 and 5 are postponed to the end of the chapter. 
The proof of item 4 relies on a generalization of Theorem \ref{theo:lattice} which we discuss now.

Let $(p_1,p_2,\dots,p_l)$ be an $l$-tuple of pairwise relatively prime integers with $1< p_1 < \dots < p_l$. 
In the case where $l=3$, it is readily known that  $\kappa(p_1,p_2,p_3)$ equals the number of lattice points in a tetrahedron. 
For $l \geq 4$, $\kappa(p_1,p_2,\dots,p_l)$ is still equal to the number of 
lattice points in a polytope, however, the polytope is not necessarily a tetrahedron. 
Another difference is that, this time each lattice point is counted with a certain multiplicity.
To state our result we need more notation. Define 
$$
N_0(k)=p_1\dots p_l\left ( k-\sum_{i=1}^l \frac{1}{p_i} \right ).
$$
Let $H_k$ be the hyperplane in $\mathbb{R}^l$ defined by

\begin{equation}\label{eq:hyperplane}
\sum_{i=1}^l\frac{x_i}{p_i}=\frac{N_0(k)}{p_1p_2\dots p_l}.
\end{equation}
%\noindent We know from Theorem \ref{theo:mainmore} that $N_0(l-2)>0$. 
Suppose $k\in \{1,2,\dots,l-2 \}$ with $N_0(k) >0$. Then $H_k$ cuts out a tetrahedron $T_k$ 
from the first orthant of $\mathbb{R}^l$. For convenience, we define $T_k = \emptyset$ if $N_0(k) < 0$. 
Then we have 
$$  
T_{l-2} \supset T_{l-3}\supset \cdots T_1 \supset T_0.  
$$
Let $C=C_{p_1,\dots,p_l}$ denote the $l$-dimensional cube $[0,p_1-1]\times[0,p_2-1]\dots\times[0,p_l-1]\subset \mathbb{R}^l$, and set    
\begin{equation}\label{d:Ak}
A_k=\# \left ( (T_k - T _{k-1}) \cap C \cap \mathbb{Z}^l \right ).
\end{equation}
In other words, $A_k$ is the number of lattice points from $C$ that lie between the hyperplanes $H_k$ and $H_{k-1}$. 
Note that $A_{l-2}>0$, and $A_j=0$ if $N_0(j)<0$.

Recall that the $\kappa$-invariant is equal to the sum of all negative values of $\Delta (n)$.
\begin{Theorem}\label{theo:genmono}
$$
\kappa(p_1,p_2,\dots,p_l)=\sum_{j=1}^{l-2}\frac{(l-j-1)(l-j)}{2}A_j.
$$
\end{Theorem}
\begin{proof}

We begin with describing a useful affine transformation on $\mathbb{R}^l$. 
Let $\varphi : \mathbb{R}^l \to \mathbb{R}^l$ denote the map defined by $\varphi (x_1,x_2,\dots,x_l)=(y_1,y_2,\dots,y_l)$, 
where $\displaystyle y_i=-\left ( x_i-(p_i-1) \right )$ for $i=1,2,\dots,l$. Clearly, $C$ is invariant under $\varphi$, and moreover,  
$\varphi$ maps the hyperplane $\widetilde{H}_k$ defined by 
$$
\sum_{i=1}^l\frac{x_i}{p_i}=l-k
$$
to the hyperplane $H_k$ defined in (\ref{eq:hyperplane}). 
Let $\widetilde{A}_k$ denote the number of lattice points in $C$ that lie between $\widetilde {H}_{k-1}$ and $\widetilde {H}_{k}$. 
Then $\widetilde{A}_k=A_k$ for every $k \in \{1,2,\dots,l-2\}$.

Next, we compute the sum of negative values of $\Delta$.
By Lemma \ref{l:deltaprop}, each $f\left (\frac{np'_i}{p_i}\right )$ term equals 
$\frac{x_i}{p_i}$ for some $x_i\in \{0,1,\dots,p_i-1\}$. 
Hence, (\ref{e:delcompgen}) implies that $\Delta (n) \geq 0$ for $n\geq (l-2)p_1\dots p_l$.
Therefore, it remains to find the sum of negative values of $\Delta(n)$ for $(k-1)p_1\dots p_l\leq n < kp_1\dots p_l$, where 
$k\in \{0,1,\dots,l-3 \}$, 

Notice that, in the interval $(k-1)p_1\dots p_l\leq n < kp_1\dots p_l$, we have 
$\Delta (n)\geq -(l-k-1)$.
Let $t\in \{ 1,2,\dots,l-k-1 \}$ and let $B(k,t)$ denote the number of times $\Delta$ attains the value $-t$ in the interval 
$(k-1)p_1\dots p_l \leq n < kp_1\dots p_l$. 
It follows from (\ref{e:delcompgen}) that $B(k,t)$ is equal to the number of integers $n$ which solve the inequality 
\begin{equation}\label{e:ineq}
t+k+1 \geq \sum_{i=1}^l f \left ( \frac{np_i'}{p_i}\right ) \geq t+k,
\end{equation}
with $(k-1)p_1\dots p_l \leq n < kp_1\dots p_l.$

Observe that the Chinese remainder theorem combined with Lemma \ref{l:deltaprop} implies that  
given any $l$-tuple $(x_1,\dots,x_l)$ with $x_i\in \{0,1,\dots,p_i-1\}$, 
there exists unique $n$ in the interval $(k-1) p_1\dots p_l \leq n < kp_1\dots p_l$ such that 
$f\left (\frac{np_i'}{p_i}\right )=\frac{x_i}{p_i}$ for $i\in \{1,\dots,l\}$. Thus,  
in the light of inequality (\ref{e:ineq}), we see that $B(k,t)$ is equal to the number of lattice points 
in $C$ that lie between $\widetilde{H}_{l-t-k}$ and $\widetilde{H}_{l-t-k-1}$. In other words $B(k,t)=\widetilde{A}_{l-t-k}$. 
Finally, by the following manipulations we finish the proof.
\begin{align*}
\kappa(p_1,p_2,\dots,p_l) &=  \left | \sum_{n=0}^\infty \mathrm{min} \{0, \Delta(n) \} \right |  
=  \left | \sum_{k=1}^{l-2}\sum _{n=(k-1)p_1\dots p_l}^{kp_1\dots p_l} \mathrm{min} \{0, \Delta(n) \} \right | \\
&=  \sum_{k=1}^{l-2}\sum _{t=1}^{l-k-1}t B(k,t) =  \sum_{k=1}^{l-2}\sum _{t=1}^{l-k-1}t \widetilde{A}_{l-t-k}\\
&=  \sum_{k=1}^{l-2}\sum _{t=1}^{l-k-1}t A_{l-t-k}=  \sum_{k=1}^{l-2}\sum _{j=1}^{l-k-1}(l-k-j) A_{j} \\
&=  \sum_{j=1}^{l-2}\sum _{k=1}^{l-j-1}(l-k-j) A_{j}\\
&=  \sum_{j=1}^{l-2}\frac{(l-j-1)(l-j)}{2} A_{j}. 
\end{align*}

\end{proof}

%We continue using the notation introduced right before Theorem \ref{theo:genmono}. 
In order for proving part 4 of Theorem \ref{theo:mainmore}, we need to relate the count of lattice points given in 
Theorem \ref{theo:genmono} to the number of lattice points in the tetrahedra $T_k$. 
This relation is established with the help of the function $\pi :\mathbb{Z}^l \rightarrow C\cap \mathbb{Z}^l$ defined by
$\pi(x_1,x_2,\dots,x_l) = (z_1,z_2,\dots,z_l)$, where $z_i\equiv x_i\; (\mathrm{mod} \; p_i)$ 
with $0\leq z_i \leq p_i-1$ for $i=1,\dots,l$.

\begin{Lemma}\label{l:cov1} Suppose $k\in \{1,2,\dots,l-2 \}$ with $N_0(k)>0$. Then 
$$
\pi \left ( (T_k\setminus T_{k-1})\cap \mathbb{Z}^l  \right ) =T_k\cap C \cap \mathbb{Z}^l.
$$
\end{Lemma}

\begin{proof}
 The inclusion  
 $$
 \pi \left ( (T_k\setminus T_{k-1})\cap \mathbb{Z}^l  \right ) \subset T_k\cap C \cap \mathbb{Z}^l
 $$
 is obvious. To prove the reverse inclusion let $(z_1,z_2,\dots,z_l)$ be a point from $T_k \cap C \cap \mathbb{Z}^l$, 
 and let $r$ be the unique non-negative integer satisfying
\begin{equation}\label{e:rremainder}
\frac{N_0(k)}{p_1p_2\dots p_l} \geq r+ \sum _{i=1}^l \frac{z_i}{p_i} > \frac{N_0(k-1)}{p_1p_2\dots p_l}.
\end{equation}
Let $r_1,\dots, r_l$ be non-negative numbers such that $r= r_1 + \cdots + r_l$. 
Define $x_i=z_i+r_ip_i$, $i=1,\dots,l$. 
Then $\pi (x_1,x_2,\dots,x_l)=(z_1,z_2,\dots,z_l)$. It follows from (\ref{e:rremainder}) that 
$$
\frac{N_0(k)}{p_1p_2\dots p_l} \geq \sum_{i=1}^l\frac{x_i}{p_i} > \frac{N_0(k-1)}{p_1p_2\dots p_l}.
$$
Hence $(x_1,x_2,\dots,x_l)\in T_k \setminus T_{k-1}$ as required.
\end{proof}

\begin{Lemma}\label{l:cov2}
Let $k\geq k'$ be two elements from $\{1,2,\dots,l-2 \}$ with $N_0(k)\geq N_0(k')>0$. 
Let $(x_1,x_2,\dots,x_l) \in (T_k \setminus T_{k-1})\cap \mathbb{Z}^l$. Then  
$\pi(x_1,x_2,\dots,x_l) \in (T_{k'} \setminus T_{k'-1})\cap C \cap  \mathbb{Z}^l$
if and only if 
$$
\sum _{i=1}^l\left \lfloor \frac{x_i}{p_i}\right \rfloor =k-k'.
$$  
\end{Lemma}

\begin{proof}

Let $(z_1,z_2,\dots,z_l)=\pi (x_1,x_2,\dots,x_l)$. Then there exist non-negative integers $r_1,r_2,\dots, r_l$ such that 
$x_i = z_i+r_ip_i$, $i=1,2,\dots,l$. Hence, 
\begin{equation}\label{e:cov1}
\sum _{i=1}^l \left \lfloor \frac{x_i}{p_i} \right \rfloor =\sum _{i=1} r_i.
\end{equation}
Then as in the proof of Lemma \ref{l:cov1}, we have that 
\begin{equation}\label{e:cov2}
\frac{N_0(k)}{p_1p_2\dots p_l} \geq \sum _{i=1}^l \frac{x_i}{p_i} > \frac{N_0(k-1)}{p_1p_2\dots p_l}.
\end{equation}
Suppose now $(z_1,z_2,\dots,z_l) \in (T_{k'} \setminus T_{k'-1})\cap C \cap  \mathbb{Z}^l$. Then 
\begin{equation}\label{e:cov3}
\frac{N_0(k')}{p_1p_2\dots p_l} \geq \sum _{i=1}^l \frac{z_i}{p_i} > \frac{N_0(k'-1)}{p_1p_2\dots p_l}.
\end{equation}
Combining  (\ref{e:cov2}) with  (\ref{e:cov3}), we see that 
$$
k-k'-1<\sum _{i=1}^l \frac{x_i- z_i}{p_i}< k-k'+1.
$$
Note that the middle term in the above inequality is an integer, and therefore, 
$$
\sum _{i=1}^l r_i =k-k'.
$$
The desired equality now follows from (\ref{e:cov1}).

Conversely, assume that $\sum_{i=1}^l r_i=k-k'$. By (\ref{e:cov2})  
$$
\frac{N_0(k)}{p_1p_2\dots p_l} \geq k-k' + \sum_{i=1}^l \frac{z_i}{p_i} > \frac{N_0(k-1)}{p_1p_2\dots p_{l-1}}.
$$
Hence
$$
\frac{N_0(k')}{p_1p_2\dots p_l} \geq \sum_{i=1}^l \frac{z_i}{p_i} > \frac{N_0(k'-1)}{p_1p_2\dots p_{l-1}},
$$
which implies $(z_1,z_2,\dots,z_l) \in (T_{k'} \setminus T_{k'-1})\cap C \cap  \mathbb{Z}^l$.

\end{proof}

\vspace{1cm}

\begin{proof}[Proof of Theorem \ref{theo:mainmore}]

Proofs of parts 1-3 are similar to those in Theorem \ref{theo:main}. 
Part 5 is a direct consequence of Equation \ref{e:delcompgen}. 
Indeed, plugging $n=p_1p_2\dots p_l\sum_{i=1}^l\frac{x_i}{p_i}$ in (\ref{e:delcompgen}), we see that
\begin{align*}
\Delta (n) &= 1+ \sum_{i=1}^l \frac{x_i}{p_i}- \sum_{i=1}^l f\left ( \frac{-x_i}{p_i}\right ) \\
&= 1- \sum_{i=1}^l \left \lceil -\frac{x_i}{p_i}\right \rceil  \\
&=1+\sum_{i=1}^l \left \lfloor \frac{x_i}{p_i}\right \rfloor \qquad \text{by}\ (\ref{d:f}).
\end{align*}
What is left is to show that $\Delta (n)$ ($n\leq N_0(l-2)$) attains all its positive values on $G$. 
Note that 
$$
\kappa  (p_1,p_2,\dots,p_l) =\left | \sum_{n=0}^\infty \min \{0, \Delta (n) \} \right | 
=\left | \sum_{n=0}^{N_0} \min \{0, \Delta (n) \} \right | =\left | \sum_{n=0}^{N_0} \max \{0, \Delta (n) \} \right |,
$$
where the second and third equalities are due to parts $2$ and $3$ respectively. 
Define
$$
\kappa ' (p_1,p_2,\dots,p_l):=\sum_{n\in G \cap [0,N_0]} \Delta (n).
$$
By part $6$, we have 
$$
\kappa '(p_1,p_1,\dots,p_l) \leq \kappa (p_1,p_1,\dots,p_l).
$$ 
Obviously, this inequality is strict if  $\Delta (n)$ attains a positive value outside of the semigroup $G$. 
We prove that this is not the case.

Using its minimal generating set, we represent the elements of $G$ by $l$-tuples as follows. 
Let $\phi : \mb{Z}^l \rightarrow G$ denote the map defined by 
$$
\phi(x_1,\dots, x_l) = p_1p_2\dots p_l \sum _{i=1}^l \frac{x_i}{p_i}.
$$
Clearly, $\phi$ is a finite-to-one map. 
Let $R(x_1,x_2,\dots,x_l)$ denote the cardinality of the set $\phi ^{-1}(\phi(x_1,x_2,\dots,x_l))$. 
We are interested in the role that $R(x_1,x_2,\dots,x_l)$ plays in the computation of $\kappa'$, rather 
than its actual value.  
\begin{eqnarray*}
\kappa ' (p_1,p_2,\dots,p_l) &=& \sum _{(x_1,x_2,\dots,x_l) \in T_{l-2} \cap \mathbb{Z}^l} 
\frac{\Delta (\phi (x_1,x_2,\dots,x_l))}{R(x_1,x_2,\dots,x_l) }\\
&=&
\underset{N_0(k) >0}{\sum _{k=1 }^{l-2}}\sum _{(x_1,x_2,\dots,x_l) \in (T_{k} \setminus T_{k-1}) \cap \mathbb{Z}^l} 
\frac{1}{R(x_1,x_2,\dots,x_l) } \left ( 1+\sum_ {i=1}^l \left \lfloor \frac{x_i}{p_i} \right \rfloor\right ). 
\end{eqnarray*}
Here, the second equality is a consequence of part \ref{i:theo6}. 
We need to count the elements appearing in the second sum.

We know from Lemma \ref{l:cov1} that 
$\pi((T_{k} \setminus T_{k-1}) \cap \mathbb{Z}^l)= T_{k} \cap C \cap \mathbb{Z}^l$, for $k=1,\dots,l-2$. 
Moreover, two $l$-tuples $(x_1,x_2,\dots,x_l)$, $(y_1,y_2,\dots,y_l)\in (T_{k} \setminus T_{k-1}) \cap \mathbb{Z}^l$ 
are mapped to the same element by $\pi$ if and only if  they are mapped onto the same element by $\phi$. 
These observations combined with Lemma \ref{l:cov2} gives 
\begin{align*}
\kappa ' (p_1,p_2,\dots,p_l) &= \underset{N_0(k) >0}{\sum _{k=1 }^{l-2}} \ \underset{N_0(j) >0}{\sum _{j=1 }^{k}} \ 
\sum _{(z_1,z_2,\dots,z_l) \in (T_{j} \setminus T_{j-1}) \cap C \cap \mathbb{Z}^l} k-j+1 \\
&= \underset{N_0(k) >0}{\sum _{k=1 }^{l-2}} \underset{N_0(j) >0}{\sum _{j=1 }^{k}} (k-j+1)A_j \qquad (\text{by}\ (\ref{d:Ak})).
\end{align*}
Note that the above equation is valid even without the restrictions $N_0(k)\geq N_0(j)>0$, since we 
force $A_j=0$, if $N_0(j)<0$. 
Changing the order of the summation and the indices accordingly give 
\begin{eqnarray*}
\kappa ' (p_1,p_2,\dots,p_l) &= & \sum _{j=1 }^{l-2}  \sum _{k=j }^{l-2} (k-j+1)A_j = \sum _{j=1 }^{l-2} \sum _{r=1 }^{l-1-j} rA_j\\
&=&  \sum _{j=1 }^{l-2} \frac{(l-j-1)(l-j)}{2} A_j.
\end{eqnarray*}
Hence by Theorem \ref{theo:genmono}, we have $\kappa ' (p_1,p_2,\dots,p_l) = \kappa (p_1,p_2,\dots,p_l)$. 

\end{proof}

Next we establish the monotonicity of $\kappa$ under the addition of one more singular fiber.
\begin{Proposition}\label{prop:kappa2}
Let $(p_1,p_2,\dots,p_l,p_{l+1})$ be an $(l+1)$-tuple of pairwise relatively prime  integers with $1<p_1<p_2<\dots<p_l<p_{l+1}$. Then
$$
\kappa(p_1,p_2,\dots,p_l)\leq \kappa (p_1,p_2,\dots,p_l,p_{l+1}).
$$
\end{Proposition}

\begin{proof}
Let $\Delta$ and $\Delta '$ be the difference terms corresponding to the $l$-tuple $(p_1,p_2,\dots,p_l)$ 
and the $(l+1)$-tuple $(p_1,p_2,\dots,p_l,p_{l+1})$ respectively. Let $n$ be an integer with $\Delta (n) \leq 0$. 
We claim that $\Delta ' (n) \leq \Delta (n)$. Writing the difference terms as in (\ref{e:deltacompact}), we have 
$$
\Delta ' (n) - \Delta (n)=\frac{n}{p_1p_2\dots p_lp_{l+1}}-\frac{n}{p_1p_2\dots p_l}-f \left ( \frac{np_{l+1}'}{p_{l+1}}  \right ) \leq 0.
$$
\end{proof}

We state the monotonicity of kappa under the natural partial order of $l$-tuples, generalizing Proposition \ref{prop:finite}. 
\begin{Proposition}\label{prop:kappa3}
Suppose $(p_1,p_2,\dots,p_l)\geq (q_1,q_2,\dots,q_l)$, then $\kappa(p_1,p_2,\dots,p_l))\geq \kappa(q_1,q_2,\dots,q_l)$.
\end{Proposition}

\begin{proof}
From the discussion preceding Theorem \ref{theo:genmono}, each  tetrahedron associated to 
$(p_1,p_2,\dots,p_l)$ is strictly larger than the corresponding tetrahedron associated to 
$(q_1,q_2,\dots,q_l)$. Hence the monotonicity follows from the count given in Theorem \ref{theo:genmono}. 
Indeed, every lattice point appearing in the calculation of $\kappa(p_1,p_2,\dots,p_l)$ appears also in the calculation of 
$\kappa(q_1,q_2,\dots,q_l)$ with a possibly bigger multiplicity. 
This is because of the fact that the lattice points in the smaller tetrahedra are counted with bigger multiplicity in Theorem \ref{theo:genmono}.
\end{proof}

\section{Topological Applications}
\label{s:topoappl}

In this section we discuss some of the topological applications of our work to the topology of $3$--manifolds. 
Our first task is to detect the Brieskorn spheres with trivial Heegaard-Floer homology. 
We would like to find all Brieskorn spheres which are $L$--spaces, so, we first translate the condition to being an 
$L$--space in terms of the tau function defined in (\ref{eq:rectau}).

\begin{Proposition}\label{p:lspace}
Let $Y$ be a $3$--manifold which bounds a negative definite plumbing with at most one bad vertex. Then 
$Y$ is an $L$--space if and only if its tau function is increasing.
\end{Proposition}

\begin{proof}
It follows from Nemethi's work that Heegaard-Floer homology in the canonical \spinc structure is given 
by the graded root associated with its tau function. In particular, this gives trivial homology if and only $\tau$ is increasing. 
Now, the proof follows from Theorem 6.3 of Nemethi \cite{N}, which states that a plumbed 3-manifold is an $L$--space if and only if 
its Heegaard-Floer homology in the canonical \spinc  structure is trivial. 
\end{proof}

It is known that the $3$-sphere and the Poincar\'e homology sphere $\varSigma (2,3,5)$ are examples of $L$-spaces. 
In fact it is conjectured that an irreducible integral homology sphere is an $L$-space if and only if 
it is homeomorphic to $S^3$, or to $\varSigma(2,3,5)$ (with either orientation). 
Here we verify this conjecture for Seifert homology spheres. 
This was observed long before by Rustamov and independently by Eftekhary, but here we give a simpler proof.

\vspace{1cm}

\begin{proof}[Proof of Theorem \ref{theo:lspace}]
In view of Proposition \ref{p:lspace}, it suffices to prove the following: If $\tau$ function of a 
Seifert homology sphere $Y:=\varSigma(p_1,p_2,\dots,p_l)$ is increasing, then either $l\leq 2$ 
(implying $Y\approx S^3$), or $l=3$ and $(p_1,p_2,p_3)=(2,3,5)$. 
That $\tau$ is increasing is equivalent to the condition that $\Delta(n)=\tau(n+1)-\tau(n) \geq 0 $. Then 
Theorem \ref{theo:mainmore} rules out the possibility that $l\geq 4$, since $\Delta (N_0)=-\Delta(0)=-1$. If $l=3$, 
Theorem \ref{theo:main} forces that $(p_1,p_2,p_3)=(2,3,5)$.
\end{proof}

\vspace{1cm}

\begin{proof}[Proof of Proposition \ref{prop:kappa}]
This is an immediate consequence of Theorem \ref{theo:nem}, and the relationship between 
Heegaard-Floer homology and the Casson invariant. More precisely, Ozsv\'ath and Szab\'o show in \cite{OS2} 
that for every integral homology sphere $Y$, 
the Heegaard-Floer homology has a decomposition of the form 
$$
HF^+(-Y)=\mathcal{T}^+_{(d)}\oplus HF_{\mathrm{red}}(-Y),
$$
where $HF_{\mathrm{red}}(-Y)$ is a finitely generated subgroup, whose Euler characteristic satisfies the following property:
\begin{equation}\label{eq:euler}
\chi(HF_{\mathrm{red}}(-Y))=\lambda(-Y)+\frac{d(-Y)}{2}.
\end{equation}
It is shown in \cite{OS1} that if $Y$ is Seifert homology sphere 
(or more generally if $Y$ bounds a negative definite plumbing with at most one bad vertex), then 
$HF^+(-Y)$ is supported only in even degrees. Hence, $\chi(HF_{\mathrm{red}}(-Y))= \mathrm{rank}(HF_{\mathrm{red}}(-Y))$ 
for every Seifert homology  sphere $Y$. 
By the discussion in Section \ref{s:prelim}, we read off this quantity from the corresponding graded root directly: 
Simply remove the longest branch, then the number of remaining vertices is the rank of $HF_{\mathrm{red}}(-Y)$. 
By Theorem \ref{theo:nem}, the graded root is determined by the tau function.

It is straightforward to verify  
$\mathrm{rank}(\mathbb{H}_{\mathrm{red}}(R_{\tau}),\chi_\tau)=\mathrm{min}_i\tau(i) + \sum_i \mathrm{max} \{-\Delta (i), 0 \}$ 
(see Corollary 3.7 of \cite{N}).
Comparing with Definition \ref{def:kappa}, we have $\kappa (p_1,\dots,p_l) = \sum_i \mathrm{max} \{-\Delta (i), 0 \}$. 
Substituting in (\ref{eq:euler}) we obtain 
$$
\kappa=\lambda(-Y)+\frac{d(-Y)}{2}-\mathrm{min}_i\tau(i).
$$
The theorem then follows from the fact that $\frac{d(-Y)}{2}-\mathrm{min}_i\tau(i)$ is the half of the degree shift term 
$(K^2+s)/4$, which is discussed in Section \ref{s:prelim}.
\end{proof}

\vspace{1cm}

\begin{proof}[Proof of Theorem \ref{theo:list}]
Using Theorem \ref{theo:main} and Nemethi's method described in Section \ref{s:prelim}, it is easy to verify Table  \ref{tab:kappa}. 
We must show that every Seifert homology sphere has $\kappa\geq 3$, except the ones given in Table \ref{tab:kappa}. 
Let $\varSigma (p_1,p_2,\dots,p_l)$ be a Seifert homology sphere that does not appear in \ref{tab:kappa}. 
Then $l\geq 3$ since only Seifert homology with less than $3$ singular fibers is $S^3$. 
Suppose $l=3$, then the triple $(p_1,p_2,p_3)$ must be greater than or equal to one of the following triples: 
$(3,5,7)$, $(3,4,7)$, $(3,5,9)$, $(2,7,9)$, $(2,5,11)$, $(2,5,13)$, $(2,5,19)$, $(2,3,19)$. 
These triples are the immediate successors of the triples appearing in the table. 
It is easy to check that all of these triples have $\kappa \geq 3$, so by monotonicity we are done in the case of three singular fibers. 
For four and more  singular fibers, we have $\kappa(p_1,p_2,\dots,p_l)\geq \kappa (2,3,5,7) \geq \kappa (3,5,7)\geq 4$.
\end{proof}

We are ready to prove Theorem \ref{theo:finite}, which states that there are only finitely many Seifert homology
spheres with a prescribed $\kappa$, and therefore, a prescribed Heegaard-Floer homology.

\begin{proof}[Proof of Theorem \ref{theo:finite}]

We already know that when $\kappa =0$, there are only two possible Seifert homology spheres, namely,
$S^3$, or the Poincar\'e homology sphere. 
For the general case, it is enough to show that $\kappa$ is not constant on any infinite family of Seifert homology spheres 
each of which contains three or more singular fibers. 

We begin with families of Brieskorn spheres. Let $\{(p_n,q_n,r_n):n=1,\dots, \infty\}$ be an infinite family of triples.  
Since $p_n<q_n<r_n$, the last entry $r_n$ can not stay constant. 
Hence, after passing to a subsequence we may assume that $(p_n,q_n,r_n)$ is increasing with $r_n\to \infty$. 
This implies that $\kappa(p_n,q_n,r_n)\to \infty$ by Proposition \ref{prop:finite} and its proof. 
In particular $\kappa(p_n,q_n,r_n)$ is not constant.
Suppose now that we have infinite family of Seifert homology spheres 
$(p_{1,n},p_{2,n},\dots,p_{l(n),n})$ with $l(n)\geq 3$ for all $n$. 
Projecting to the last three coordinates and using Proposition \ref {prop:kappa2}, we get an infinite family of triples 
$(p_n,q_n,r_n)$ such that $\kappa (p_{1,n},p_{2,n},\dots,p_{l(n),n})\geq \kappa(p_n,q_n,r_n)$. 
As before, we may assume that $\kappa(p_n,q_n,r_n)\to \infty$, and hence, $\kappa (p_{1,n},p_{2,n},\dots,p_{l(n),n})\to \infty$. 

To finish our argument we need to know that every positive integer can be realized as $\kappa$ 
of some Seifert homology sphere. Indeed, one can directly verify from Theorem \ref{theo:main} that $\kappa(2,3,6k+1)=k$.
Hence, the proof is complete. 
 
%The claim about the finiteness of any family of isomorphic Heegaard-Floer homology follows from 
%Nemethi's work explained in Section \ref{s:prelim}.

\end{proof}

\section{Weakly Elliptic Brieskorn Spheres}
\label{s:weakly elliptic}

In this section we use our findings to characterize all weakly elliptic Brieskorn spheres $\varSigma(p,q,r)$ 
in terms of their defining integers $1< p < q <r$. We begin with introducing a new concept on numerical semigroups.

\begin{Definition}
Let $G$ be a numerical semigroup and let $n_0\in \N-G$ be a positive integer. Then $G$ is said to  
{\em alternate with respect to $n_0$}, if for every $x,y\in G$ such that $x<y<n_0$, there exists $z\in G$ satisfying 
$x< n_0 - z < y$. 
\end{Definition}

Note that if $G$ is generated by a single element $a$, then $G$ alternates with respect to any $n_0\in \N - G$.
This notion gets more interesting if there are more than one generators. Clearly, in this case, there are only finitely many 
possibilities for $n_0$. 

\begin{Lemma}
Let $G=G(a,b,c)$ be a numerical semigroup minimally generated by three relatively prime positive integers $a+1<b<c$, 
and let $n_0$ be a number from $\N - G$.
Then $G$ alternates with respect to $n_0$ if and only if $a < n_0 < b<c$. 
\end{Lemma}

\begin{proof}
($\Leftarrow$) Our claim is immediately proven once we replace $G(a,b,c)$ by $G(a)$.

($\Rightarrow$) Let $n_0 \in \N - G$ be a positive integer with respect to which $G$ alternates. 
Clearly, if $n_0<a$, then there is nothing to prove. 
We proceed by induction on $n_0$, the base case being $n_0 = a+1$. Notice that our claim is trivially 
true in the base case. 

Assume now that if $n_0'<n_0$ and $G$ is alternating with respect to $n_0'$, then $a<n_0'<b<c$.
Suppose $x<y$ are from $G$ and they are the largest elements of $G$ that are less than $n_0$. 
Thus, there exists $z\in G$ such that $x < n_0-z < y<n_0$. It follows that  $x+z < n_0 < y+z < n_0 +z$, hence $x+z =y$. 
Notice that $z$ has to be the smallest element $a$ of $G$, otherwise, for $w\in G$ with $w<z$ we see that 
$x < w+x < y$, contradicting with the maximality of $x$.
 
We claim that $G$ alternates with respect to $n_0'=n_0-z$. Indeed, $n_0-z \notin G$ and if $u < v$ are two elements 
from $G$ such that $u < v < n_0 -z$, then $u+z < v+z < n_0$, hence there exists $w\in G$ such that $u+z+w < n_0 < v+z +w$. 
Our claim follows from this. 

Now, by induction hypothesis we have that $a < n_0 - z < b<c$. But $x< n_0-z$, so $x$ must be a multiple of $a$.
Then $y=x+z$ is a multiple of $a$. If $n_0 < b+z < y+z$, then $x<b < n_0$. Since $x$ is the second largest element of $G$ 
that is less than $n_0$, and since $b$ is not a multiple of $a$, we obtained a contradiction. Therefore, $y+z < b+z$, or $y<b$. 
This implies that $n_0 < b$ and the proof is finished. 

\end{proof}

\begin{Corollary}\label{C:first criterion}
Let $1<p<q<r$ be three relatively prime integers. Then the Brieskorn sphere $\varSigma(p,q,r)$ 
is weakly elliptic if and only if $N_0 < pr$, where $N_0 = pqr-pq-pr-qr$. 
\end{Corollary}
\begin{proof}

It follows from the discussion in Section \ref{s:prelim} that $\varSigma= \varSigma(p,q,r)$ is weakly elliptic if and only if its difference function 
$\Delta_\varSigma$ alternates along its non-zero entries in the domain $[0,N_0]$. 
Interpreting in terms of the numerical semigroup $G_\varSigma= G(pq,pr,qr)$ of $\varSigma$, we see that 
if $pq < N_0 < pr$, $\Delta_\varSigma$ alternates with respect to $N_0$ 
if and only if $G_\varSigma$ alternates with respect to $N_0$.
On the other hand, if $0< N_0 < pq$, there is nothing to prove, because there are only two non-zero values of $\Delta_\varSigma$
in $[0,N_0]$ and these are $1$ and $-1$.
\end{proof}

\begin{proof}[Proof of Theorem \ref{T:complete list}]
($\Rightarrow$)
Let $\varSigma(p,q,r)$ be a weakly elliptic Brieskorn sphere. By Corollary \ref{C:first criterion}, we know that 
$pqr - pq -pr -qr < pr$. Dividing by $pqr$, we obtain 
\begin{align}\label{reduction to fraction}
1 - \frac{1}{r} - \frac{1}{q} - \frac{1}{p} < \frac{1}{q}.
\end{align}
Since $1 < p < q < r$, it follows that $1- 3/p< 1/q$, or $1 < 1/q+3/p$, which implies $1 < 4/p$. 
Thus, we conclude that $p <4$.

We proceed with the case $p=3$. Using (\ref{reduction to fraction}) we see that 
$2/3 - 1/r < 2/q$. Hence, if $r\geq 6$, then $2/3 - 1/6 \leq 2/3 -1/r < 2/q$. In other words, $1/2 < 2/q$, or $q < 4$, which
is a contradiction. Therefore, $r<6$, hence the only possibility is that $q=4$ and $r=5$.

Next, we look at the case when $p=2$. Then we have 
\begin{align}\label{reduction to fraction 2}
 \frac{1}{2} - \frac{1}{r} - \frac{1}{q} < \frac{1}{q}.
\end{align}
This inequality implies that $q< 6$. There are two possibilities, $q=3$ and $q=5$. In the former case, we are done, already.
For the latter, it follows from (\ref{reduction to fraction 2}) that $r <10$. Obviously, the only two possibilities are $r=7$ and $r=9$.

$(\Leftarrow)$ It follows from the definition of weakly elliptic Brieskorn spheres 
and Table \ref{tab:kappa} that $\varSigma(2,5,7)$, $\varSigma(2,5,9)$, $\varSigma(3,4,5)$, $\varSigma(2,3,5)$, 
$\varSigma(2,3,7)$, and $\varSigma(2,3,13)$ are 
weakly elliptic. Therefore,  it is enough to show that $\varSigma(2,3,r)$, $r> 13$ is weakly elliptic. 

Notice that any integer $r > 13$ that is relatively prime to $2$ and $3$ has the form $r= 6k \pm 1$ for some $k \geq 3$.
We proceed with the case that $r= 6k+1$. Then $N_0 = 6k - 5$. It follows that $6 < N_0 < 2(6k+1) < 3(6k+1)$, if $k\geq 3$. 
Therefore, by Corollary \ref{C:first criterion}. $\varSigma(2,3,6k+1)$ is weakly elliptic. 
In the next case that $r=6k-1$, we have $N_0 = 6k-7$. Similar to the previous case, $6 < N_0 < 2(6k-1)$, if $k \geq 3$.
Therefore, $\varSigma(2,3,6k-1)$ is weakly elliptic and the proof in the case of Brieskorn spheres is finished. 

Finally, for more than three singular fibers, we observe that the statement and the proof of Corollary \ref{C:first criterion} 
is valid if $N_0 < p_1p_3\cdots p_l$. However, an argument similar to ``if'' part of the proof of three singular fibers gives 
a contradiction to this inequality.

\end{proof}

\section{Generating Function of $\tau$}
\label{s:genfunc}

In this section we calculate the generating functions for the sequences $\tau (n)$ and $\Delta (n)$. 
Our main result shows that both generating functions are rational. 
For convenience we change our notation slightly. 
Let $\alpha = m /a = m_1/a_1$, $\beta= m_2/a_2$ and $\gamma= m_3/a_3$ be three rational numbers. 
Consider the integer valued function defined by the recurrence relation
\begin{align}\label{A:recurrence of tau}
\tau (n+1) 
= \tau (n) + 1 + |e_0|n - \left\lceil \frac{n}{\alpha}  \right\rceil - \left\lceil \frac{n}{\beta}  \right\rceil - \left\lceil \frac{n}{\gamma}  \right\rceil,
\end{align}
and the initial condition $\tau(0)=0$. 
%In this section we prove  the following theorem which determines the generating function of $\tau$.

\begin{Theorem}\label{T: First Theorem}
Let $(m_1,a_1),(m_2,a_2),(m_3,a_3) $ be three pairs of pairwise relatively prime positive integers, and 
let $\tau :\N \rightarrow \Z$ denote the $\tau$-function defined recursively as in  (\ref{A:recurrence of tau}).
Then its generating series $F(x) = \sum_{n\geq 1} \tau(n) x^n$ is given by 
$$
F(x) = \frac{x}{(1-x)^2} + \frac{|e_0|x^2}{(1-x)^3} - \frac{x^2}{(1-x)^2}\sum_{i=1}^3\frac{(1-x^{\lfloor m_i/a_i \rfloor a_i})}{(1-x^{m_i})(1-x^{\lfloor m_i/a_i \rfloor})},
$$
where $\lfloor y \rfloor$ denotes the \emph{floor function}.
\end{Theorem}

Theorem  \ref{T: First Theorem} immediately implies Theorem \ref{C:simplified}. 
The proof of Theorem \ref{T: First Theorem} occupies the rest of this subsection. 
The main component of the proof is the identification of  the generating function 
$f(x) = \sum_{n \geq 0} \left\lceil \frac{n a}{m}  \right\rceil x^n$ 
with a simple rational function. We achieve this in two steps. 

\begin{Lemma}\label{L:prep1}
Let $D(x)$ denote the polynomial $D(x):= \sum_{i=1}^{m-1} \lf \frac{ia}{m} \rf x^i$. Then 
$$
f(x) = \frac{ax^m+ D(x) (1-x) }{(1-x) (1-x^m)}.
$$
\end{Lemma}

\begin{proof}
To compute $f(x)$ in a closed form we break it into congruence classes modulo $m$ (without worrying about convergence issues):
\begin{align*}
f(x) &= 
\sum_{n \equiv 0 \mod m} \left\lceil \frac{n a}{m}  \right\rceil x^{n} + 
\sum_{n \equiv 1 \mod m} \left\lceil \frac{n a}{m}  \right\rceil x^{n}+ \cdots + 
\sum_{n \equiv m-1 \mod m} \left\lceil \frac{n a}{m}  \right\rceil x^{n},
\end{align*}
or
\begin{align}\label{A:expansion of f}
f(x) &= 
\sum_{l\geq 0} \left\lceil \frac{lm a}{m}  \right\rceil x^{lm} + 
\sum_{l \geq 0} \left\lceil \frac{(lm+1) a}{m}  \right\rceil x^{lm+1} + \cdots +
\sum_{l \geq 0} \left\lceil \frac{(lm+m-1) a}{m}  \right\rceil x^{lm+m-1}. 
\end{align}
Note that, for $i=1,\dots, m-1$ 
\begin{align*}
\sum_{l \geq 0} \left\lceil \frac{(lm+i) a}{m}  \right\rceil x^{lm+i} =
\sum_{l \geq 0} \left( la+ \left\lceil \frac{ia}{m} \right\rceil \right) x^{lm+i}.
\end{align*}
We separate the right hand side of (\ref{A:expansion of f}) into two summations; $f(x) = A(x) + B(x)$, where 
\begin{align*}
A(x) = \sum_{i=0}^{m-1} \sum_{l \geq 0} al x^{ml+i} \qquad \text{and}\qquad B(x) = \sum_{i=1}^{m-1} \sum_{l\geq 0} \left\lceil \frac{ia}{m} \right\rceil x^{ml+i}.
\end{align*}
It is easier to find a closed formula for $A(x)$; 
\begin{align*}
A(x) = \sum_{i=0}^{m-1} \sum_{l \geq 0} al x^{ml+i} &= \sum_{i=0}^{m-1} x^i \sum_{l \geq 0} al x^{ml} \\
&= \frac{1- x^m}{1-x} a \sum_{l \geq 0} l (x^m)^l = \frac{1- x^m}{1-x} a x^m \frac{1}{(1-x^m)^2}= \frac{ax^m}{(1-x) (1-x^m)}.
\end{align*}
For $B(x)$ we have 
\begin{align*}
B(x) =\sum_{i=1}^{m-1} \sum_{l\geq 0} \left\lceil \frac{ia}{m} \right\rceil x^{ml+i}  
= \left( \sum_{i=1}^{m-1} \lf \frac{ia}{m} \rf x^i \right)  \sum_{l\geq 0} x^{ml} 
= \left( \sum_{i=1}^{m-1} \lf \frac{ia}{m} \rf x^i \right) \frac{1}{1-x^m}.
\end{align*}
Thus, if we define $D(x)$ as in hypothesis, 
\begin{align*}
f(x) &= A(x) + B(x) 
= \frac{ax^m}{(1-x) (1-x^m)}+  D(x) \frac{1}{1-x^m} 
= \frac{ax^m+ D(x) (1-x) }{(1-x) (1-x^m)}.
\end{align*}

\end{proof}

\begin{Lemma}\label{L: en basit hali}
Let $m=pa+q$ with $0\leq q < a$ then $f(x)$  can be written as
$$f(x)=\frac{x(1-x^{pa})}{(1-x)(1-x^m)(1-x^p)}$$
\end{Lemma}

\begin{proof}

From Lemma \ref{L:prep1},
$$
f(x)=(\sum_{i=0}^{m-1}c_ix^{i+1}) \frac{1}{(1-x)(1-x^m)},
$$
where
\begin{equation*}
c_i=\left \lceil \frac{(i+1)a}{m} \right \rceil -\left \lceil \frac{ia}{m} \right \rceil = 
\left \{
\begin{array}{ll}
1 & \mathrm{if}\;i\equiv 0\; (\mathrm{mod}\;p)\\
0 & \mathrm{otherwise} \\
\end{array} 
\right .
\end{equation*}
Therefore
\begin{align*}
f(x)=\frac{\sum_{j=1}^{a-1}x^{jp+1}}{(1-x)(1-x^m)}
=\frac{x\sum_{j=1}^{a-1}(x^{p})^j}{(1-x)(1-x^m)}
=\frac{x(1-x^{pa})}{(1-x)(1-x^m)(1-x^p)}.
\end{align*}
\end{proof}

\vspace{1cm}

\begin{proof}[Proof of Theorem \ref{T: First Theorem}]
Let $F(x)$ and $f_i(x)$ for $i=1,2,3$ denote the generating functions of $\tau(n)$ and $\left\lceil \frac{n a_i}{m_i}  \right\rceil $, respectively. 
If we multiply both sides of the equation
\begin{align*}
\tau (n+1) = \tau (n) + 1 + |e_0|n - \left\lceil \frac{n}{\alpha}  \right\rceil - \left\lceil \frac{n}{\beta}  \right\rceil - \left\lceil \frac{n}{\gamma}  \right\rceil
\end{align*}
by $x^{n+1}$ and sum over $n\geq 0$, then we obtain 
$$
F(x)  =  x F(x) + \sum_{n\geq 0} x^{n+1} +\sum_{n\geq 0} |e_0|n x^{n+1} - x (f_1(x) + f_2(x) + f_3(x) ).
$$
Equivalently, 
$$
F(x) = \frac{1}{1-x} \left( \frac{x}{1-x} + \frac{|e_0|x^2}{(1-x)^2}  - x (f_1(x) + f_2(x) + f_3(x))  \right).
$$
Therefore, the result follows from Lemma \ref{L: en basit hali}.
\end{proof}

\subsection{Closed form of $\tau (n)$}

In this subsection we use the generating function given in Theorem \ref{T: First Theorem} to find $\tau$ explicitly. 
See \cite{BN} for an alternative formula in terms of Dedekind sums.

\begin{Theorem}\label{t:explicittau}
The unique solution to the recurrence defined in (\ref{A:recurrence of tau}) is given by
\begin{eqnarray*}
\tau (n)=n&+&|e_0|\left ( \frac{n(n-1)}{2} \right )\\
&+&\sum _{i=1}^3\sum _{k=0}^{a_i-1}\left (-(n- \lfloor m_i/a_i \rfloor k-1)+\frac{m_i}{2}\left \lfloor \frac{n- \lfloor m_i/a_i \rfloor k-1}{m_i} \right \rfloor \right )  \left ( \left \lfloor \frac{n-\lfloor m_i/a_i \rfloor k-1}{m_i} \right \rfloor + 1 \right ).
\end{eqnarray*}
\end{Theorem}

Before starting the proof we first we state a useful lemma whose proof is omitted. 

\begin{Lemma}\label{L: carpma}
If $g(x)=\sum_{n=0}^\infty c_nx^n$, then 
$\displaystyle{\frac{x}{(1-x)^2}g(x)=\sum_{n=0}^\infty \left (\sum_{j=0}^n(n-j)c_j \right )x^n}.$

\end{Lemma}

\begin{proof}[Proof of Theorem \ref{t:explicittau}]
Let $p_i:=\lfloor m_i/a_i \rfloor$ for all $i=1,2,3$. By Theorem \ref{T: First Theorem},
\begin{eqnarray*}
\sum_{n=0}^\infty \tau (n) x^n &=& \frac{x}{(1-x)^2} \left ( 1+\frac{|e_0|x}{1-x} - \sum_{i=1}^3 \sum_{k=0}^{a_i-1}\frac{x^{kp_i+1}}{1-x^{m_i}}\right )\\
&=& \frac{x}{(1-x)^2} \left ( 1+\sum_{n=1}^\infty |e_0| x^n - \sum_{i=1}^3 \sum_{k=0}^{a_i-1}\sum_{n=0}^\infty x^{m_in+kp_i+1}\right )\\
&=& \frac{x}{(1-x)^2} \left ( \sum_{n=0}^\infty \xi (n) x^n - \sum_{n=0}^\infty\sum_{i=1}^3 \sum_{k=0}^{a_i-1}\epsilon_{m_i}(n-kp_i-1)x^n\right )\\
&=& \frac{x}{(1-x)^2} \left ( \sum_{n=0}^\infty c_n x^n \right ),
\end{eqnarray*}
where $c(n)=\xi (n)-\sum_{i=1}^3 \sum_{k=0}^{a_i-1}\epsilon_{m_i}(n-kp_i-1)$, and

$$
\epsilon _m(j)=\left \{
\begin{array}{ll}
1 & \mathrm{if}\;j\equiv 0\; (\mathrm{mod}\;m)\\
0 & \mathrm{otherwise} \\
\end{array} 
\right .\;\;\;\;
\xi (n)=\left \{
\begin{array}{ll}
0& \mathrm{if}\;n=0\\
|e_0| & \mathrm{otherwise} \\
\end{array} 
\right .
$$
By Lemma \ref{L: carpma},
\begin{eqnarray*}
\tau (n)&=&\sum_{j=0}^n (n-j) c_j\\
&=&\sum_{j=0}^n(n-j)\xi(j)-\sum_{i=1}^3\sum_{k=0}^{a_i-1}\sum_{j=0}^n (n-j) \epsilon_{m_i}(j-kp_i-1)\\
&=& n + \sum _{j=1}^n (n-j)|e_0|-\sum_{i=1}^3\sum_{k=0}^{a_i-1}\sum_{j=0}^{ \lfloor (n-kp_i-1)/m_i \rfloor} n-(m_ij+kp_i+1)\\
&=& n + |e_0|n^2- \frac{|e_0|n(n+1)}{2}-\sum_{i=1}^3\sum_{k=0}^{a_i-1} \left ( (n-kp_i-1)-\frac{m_i}{2}\left ( \left \lfloor \frac{n-kp_i-1}{m_i}\right \rfloor \right )\right )\left ( \left \lfloor \frac{n-kp_i-1}{m_i}\right \rfloor +1 \right ).
\end{eqnarray*}
Hence the proof is complete.
\end{proof}

\vspace{1cm}

\noindent \textbf{Acknowledgements.}
This collaboration has begun during 2012 G\"okova Geometry-Topology Conference. 
Authors would like thank the organizers for providing such a stimulating atmosphere.
Authors are grateful to Tye Lidman for helpful discussions. 
The second author is supported by the National Science Foundation Grant DMS-1065178, and by a Simons fellowship.

%%%% END OF DOCUMENT %%%%

\bibliography{References}
\bibliographystyle{plain}

\end{document}